\documentclass[runningheads,envcountsect,envcountsame]{llncs}
\usepackage{amsmath}
\usepackage{amssymb}
\usepackage{bm}
\usepackage{subfigure}
 \usepackage[hidelinks]{hyperref}
 \usepackage[disable]{todonotes}
 \usepackage[shortlabels]{enumitem}

\usepackage{graphicx}
 \usepackage[appendix=inline]{apxproof}
\ifx\proof\inlineproof
  \def\apxmark{}
\else
  \def\apxmark{\,*$\!$}
\fi
\newtheoremrep{propositionx2}[theorem]{Proposition}
\newtheoremrep{proposition2}[theorem]{Proposition\apxmark}
\newtheoremrep{theorem2}[theorem]{Theorem\apxmark}
\newtheoremrep{lemma2}[theorem]{Lemma\apxmark}
\newtheoremrep{claim2}[theorem]{Claim\apxmark}
\def\subqed{{\color{gray}\tiny\qed}}

\newcommand{\ca}[1]{\mathcal{#1}}
\newcommand{\mx}[1]{\boldsymbol{#1}}
\newcommand{\mpp}{\mathcal{P}}
\newcommand{\mcc}{\mathcal{C}}
\newcommand{\baa}{\boldsymbol{A}}
\newcommand{\bbb}{\boldsymbol{B}}
\newcommand{\bvi}{\boldsymbol{V_i}}
\newcommand{\bvj}{\boldsymbol{V_j}}
\newcommand{\bcij}{\boldsymbol{C_{ij}}}
\newcommand{\bss}{\boldsymbol{S}}

\newcommand{\bbij}{\boldsymbol{B_{ij}}}
\newcommand{\boo}{\boldsymbol{O}}
\newcommand{\bci}{\boldsymbol{C_i}}
\newcommand{\cupdot}{\mathbin{\mathaccent\cdot\cup}}

\begin{document}
\title{Twin-width and Transductions of Proper $k$-Mixed-Thin Graphs%
	\thanks{Supported by the {Czech Science Foundation}, project no.~{20-04567S}.}}
\author{Jakub Balab\'an\orcidID{0000-0002-2475-8938} \and
Petr Hlin\v{e}n\'y\orcidID{0000-0003-2125-1514} \and
Jan~Jedelsk\'y\orcidID{0000-0001-9585-2553}}
\authorrunning{J. Balab\'an et al.}
\institute{Faculty of Informatics, Masaryk University,\\Botanick\'a 68a, Brno, Czech Republic}
\maketitle              
\begin{abstract}
The new graph parameter twin-width, introduced by Bonnet, Kim, Thomass\'e and Watrigant in 2020, 
allows for an FPT algorithm for testing all FO properties of graphs.
This makes classes of efficiently bounded twin-width attractive from the algorithmic point of view.
In particular, classes of efficiently bounded twin-width include proper interval graphs, and (as digraphs) posets of width~$k$.
Inspired by an existing generalization of interval graphs into so-called $k$-thin graphs, 
we define a new class of {\em proper $k$-mixed-thin graphs} which largely generalizes proper interval graphs.
We prove that proper $k$-mixed-thin graphs have twin-width linear in~$k$, and that 
a slight subclass of $k$-mixed-thin graphs is transduction-equivalent to posets of width~$k'$
such that there is a quadratic-polynomial relation between $k$ and~$k'$.
In addition to that, we also give an abstract overview of the so-called red potential method
which we use to prove our twin-width bounds.

\end{abstract}

\section{Introduction}\label{sec:intro}

The notion of twin-width (of simple graphs, digraphs, or matrices) was introduced quite recently, in 2020, by Bonnet, Kim, Thomass{\'{e}}
and Watrigant~\cite{DBLP:conf/focs/Bonnet0TW20}, and yet has already found many very interesting applications.
These applications span from efficient parameterized algorithms and algorithmic metatheorems, through finite model theory, to classical combinatorial questions.
See also the (still growing) series of follow-up papers
\cite{DBLP:journals/jacm/BonnetKTW22,DBLP:conf/icalp/BergeBD22,DBLP:conf/soda/BonnetGKTW21,DBLP:conf/icalp/BonnetG0TW21,DBLP:conf/stoc/BonnetGMSTT22,DBLP:conf/soda/BonnetKRT22,DBLP:journals/corr/abs-2204-00722,DBLP:journals/corr/abs-2102-06880}.

We leave formal definitions for the next section.
In simple graphs, twin-width measures how diverse the neighbourhoods of the graph vertices are.
Specially, if two vertices $x$ and $y$ in a graph $G$ have the same neighbours in $V(G) \setminus \{x,y\}$, then $x$ and $y$ are called {\em twins}.
E.g., {\em cographs} (the graphs which can be built from singleton vertices by repeated operations of a disjoint union and taking the complement)
have the lowest possible value of twin-width,~$0$, which means that they can be brought down to a single vertex by successively identifying twins.
Hence the name, {\em twin-width}, for the parameter, and the term {\em contraction sequence} referring to the described identification process of vertices.

Twin-width is particularly useful in the algorithmic metatheorem area.
Namely, Bonnet et al.~\cite{DBLP:journals/jacm/BonnetKTW22} proved that classes of binary relational structures 
(such as graphs and digraphs) of bounded twin-width have efficient first-order (FO) model checking algorithms, given a witness of the boundedness (a ``good'' contraction sequence).
In one of the previous studies on algorithmic metatheorems for dense structures, Gajarsk\'y et al.~\cite{DBLP:conf/focs/GajarskyHLOORS15} proved
that posets of bounded width (the {\em width of a poset} is the maximum size of an antichain) admit efficient FO model checking algorithms.
In this regard, \cite{DBLP:journals/jacm/BonnetKTW22} generalizes \cite{DBLP:conf/focs/GajarskyHLOORS15} since posets of bounded width have bounded twin-width.
The original proof of the latter in~\cite{DBLP:journals/jacm/BonnetKTW22} was indirect (via so-called mixed minors, but this word `mixed' has nothing to do with our `mixed-thin') and giving a loose bound, 
and Balab\'an and Hlin\v{e}n\'y~\cite{DBLP:conf/iwpec/BalabanH21} have recently proved a straightforward linear upper bound (with an efficient construction of a contraction sequence) on the twin-width of posets in terms of width.

Another well-known class of graphs with bounded twin-width are proper interval graphs, also known as unit interval graphs (in contrast, the twin-width of general interval graphs is known to be unbounded \cite{DBLP:conf/soda/BonnetGKTW21}). In this paper, we significantly generalize proper interval graphs to a new class which is still of bounded twin-width. We call this new class {\em proper $k$-mixed-thin graphs} (where $k\in \mathbb{N}^+$, see Definition~\ref{def:mixThin}) since it is related
to previous generalizations of interval graphs to thin~\cite{DBLP:journals/orl/ManninoORC07} and proper thin~\cite{DBLP:journals/dam/BonomoE19} graphs.
We show some basic properties and relations of our new class, 
and prove that the twin-width of any proper $k$-mixed-thin graph is at most linear in~$k$.
Moreover, a contraction sequence can be constructed efficiently if a proper mixed-thin representation of the graph is given.
This result brings new possibilities of proving boundedness of twin-width for various graph classes in a direct and efficient way.
The aspect of an efficient construction of the relevant contraction sequence is quite important from the algorithmic point of view;
the exact twin-width is \textsc{NP}-hard to determine~\cite{DBLP:conf/icalp/BergeBD22}, and no efficient approximations of it are known in general.

The linear bound on twin-width of proper $k$-mixed-thin graphs is obtained using a natural combinatorial argument. 
Informally, we choose a subset of all possible contractions, we assign a value to each of them measuring how ``bad'' choice it is,
and then we argue that the average of these values is always ``good enough''.
Since there obviously is a choice of a contraction with the assigned value at most equal to this average, at each step, 
we can apply such contraction as the next step of our constructed contraction sequence. 
The same proof technique has been used already in \cite{DBLP:conf/iwpec/BalabanH21} to bound the twin-width of posets, 
and so it is natural to ask how far can it be generalized to efficiently bound the twin-width of other classes of not only graphs.
That is why we call this technique the {\em red potential method}, where the words {\em red potential} denote the value assigned as above
to each considered contraction.
We formulate and study this generalized concept later in the paper.

The second point of interest of our research stems from the following deep result of \cite{DBLP:journals/jacm/BonnetKTW22}:
the property of a class to have bounded twin-width is preserved under {\em FO transductions} which are, roughly explaining, 
expressions (or logical {\em interpretations}) of another graph in a given graph using formulae of FO logic with help of arbitrary additional parameters in the form of vertex labels.
E.g., to prove that the class of interval graphs has unbounded twin-width, it suffices to show that they interpret in FO all graphs.
In this regard we prove that a subclass of our new class, of the {\em inversion-free} proper $k$-mixed-thin graphs, is transduction-equivalent to the class of posets of width~$k'$ (with a quadratic dependence between $k$ and~$k'$).
So, our results can be seen as a generalization of \cite{DBLP:conf/iwpec/BalabanH21} and, importantly for possible applications, they target undirected graphs instead of special digraphs in the poset case.

\vspace*{-1ex}
\subsection{Outline of the paper}
\begin{itemize}\parskip2pt
\item[$\diamond$] In Section \ref{sec:prel} we give an overview of the necessary concepts from graph theory and FO logic;
	namely about intersection graphs, the twin-width and its basic properties, and FO transductions.

\item[$\diamond$] In Section \ref{sec:def} we define the new classes of $k$-mixed-thin and proper $k$-mixed-thin graphs, and their inversion-free subclasses
	(Definition~\ref{def:mixThin}).
\\	We also state the following results;
\begin{itemize}
\item[--]comparing proper $k$-mixed-thin to $k$-thin graphs (Propositions \ref{prop:oldthin} and \ref{prop:easymix}),
\item[--]proving that multidimensional\ full grids (i.e., strong products of paths), and 
	the proper intersection graphs of subpaths in a subdivision of a given graph, are proper $k$-mixed-thin for suitable~$k$
	(Theorems \ref{thm:multidimgrid} and \ref{thm:pathsInGraphs}).
\end{itemize}

\item[$\diamond$] Section \ref{sec:tww} brings the first core result composed of
\begin{itemize}
\item[--] an efficient constructive proof that the class of proper $k$-mixed thin graphs has twin-width at most~$9k$ (Theorem~\ref{thm:propermix-to-tww}),
	and an example in which this bound cannot be improved below a linear function (Proposition~\ref{prop:propermix-to-tww-lower}),
\item[--] followed by a consequence that FO properties on these graphs can be tested in FPT, given the representation (Corollary~\ref{cor:propermix-FOmc}).
\end{itemize}

\item[$\diamond$] Section \ref{sec:trans} then states the second core result -- the transduction equivalence.
\begin{itemize}
\item[--] The class of inversion-free proper $k$-mixed-thin graphs is a transduction of the class of posets of width at most $5\cdot \binom{k}{2} + 2k$ (Theorem~\ref{theorem:from-posets}), and
\item[--] the class of posets of width at most $k$ is a transduction of the class~of inversion-free proper $(2k+1)$-mixed-thin unordered graphs (Theorem~\ref{theorem:to-posets}).
\end{itemize}

\item[$\diamond$] In Section~\ref{sec:red-potential} we give a general overview of the red potential method generalizing the arguments in Section~\ref{sec:tww},
and study the question of how far this method can be extended to cover other classes of bounded twin-width.

\item[$\diamond$] We conclude our findings, state some open questions and outline future research directions in the final Section~\ref{sec:conclu}.
\end{itemize}

\ifx\proof\inlineproof\else
\noindent
We leave proofs of the \apxmark~-marked statements for the full preprint~\cite{DBLP:journals/corr/abs-2202-12536}.
\fi

\section{Preliminaries and Formal Definitions}\label{sec:prel}

A (simple) \emph{graph} is a pair $G=(V,E)$ where $V=V(G)$ is the \emph{finite} vertex set and $E=E(G)$ is the edge set -- a set of unordered pairs of vertices $\{u,v\}$, shortly~$uv$.
For a set $Z\subseteq V(G)$, we denote by $G[Z]$ the subgraph of $G$ induced on the vertices of~$Z$.
A \emph{subdivision} of an edge $uv$ of a graph $G$ is the operation of replacing $uv$ with a new vertex $x$ and two new edges $ux$ and~$xv$.

A {\em poset} is a pair $P=(X,\le)$ where the binary relation $\leq$ is an ordering on~$X$.
We represent posets also as special digraphs (directed graphs with ordered edges).
The \textit{width of a poset} $P$ is the maximum size of an antichain in $P$, i.e., the maximum size of an independent set in the digraph $P$.
We say that $(x, y) \in X^2$ is a \textit{cover pair} if $x \lneq y$ and there is no $z \in X$ such that $x \lneq z \lneq y$.

\vspace*{-1ex}
\subsection{Intersection graphs}
\vspace*{-1ex}

The \emph{intersection graph} $G$ of a finite collection of sets $\{S_1, \dots, S_n\}$ is a graph in which each set
$S_i$ is associated with a vertex $v_i \in V(G)$ (then $S_i$ is the {\em representative} of~$v_i$), 
and each pair $v_i,v_j$ of vertices is joined by an edge if and only if the corresponding sets have a non-empty intersection, i.e.~$v_iv_j \in E(G) \iff S_i \cap S_j \neq \emptyset$.
We say that an intersection graph $G$ is {\em proper} if $G$ is the intersection graph of $\{S_1, \dots, S_n\}$ such that $S_i\not\subseteq S_j$ for all $i\not=j\in\{1,\ldots,n\}$.

A nice example of intersection graphs are {\em interval graphs}, which are the intersection graphs of intervals on the real line.
More generally, for a fixed graph~$H$, if $H'$ is a subdivision of $H$, then an {\em$H$-graph} is the intersection graph of the vertex sets of connected subgraphs of $H'$.
Such an intersection representation is also called an $H$-representation.
For instance, interval graphs coincide with $K_2$-graphs.
We can speak also about proper interval or proper $H$-graphs.

\vspace*{-1ex}
\subsection{Twin-width}\label{sub:twwdef}
\vspace*{-1ex}

We present the definition of twin-width focusing on matrices, as taken from~\cite[Section~5]{DBLP:conf/focs/Bonnet0TW20}. 
Later in the paper, we will restrict ourselves only to the {\em symmetric} twin-width because the more general version is not relevant for graphs.

Let $\baa$ be a square matrix with entries from a finite set (here $\{0,1, r\}$ for graphs) and let $X$ be the set indexing both rows and columns of~$\baa$.
The entry $r$ is called a {\em red entry}, and the \textit{red number} of a matrix $\baa$ is the maximum number of red entries over all columns and rows in $\baa$.

\begin{figure}[tb]
\begin{subfigure}
\centering\includegraphics[page=1]{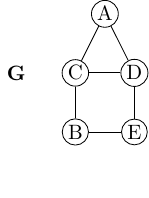}
\end{subfigure}
\hfill
\begin{subfigure}
        \centering\includegraphics[page=1]{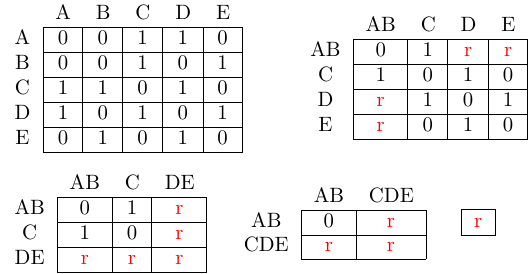}
\end{subfigure}
\caption{An example of a graph $G$ (left), and a symmetric contraction sequence of its adjacency matrix (right), which certifies that the symmetric twin-width of the adjacency matrix of $G$ is at most 3, and so is the twin-width of $G$.}
        \label{figure:tww}
\end{figure}

\emph{Contraction} of the rows (resp. columns) $k$ and $\ell$ results in the matrix obtained by deleting $\ell$, 
and replacing entries of $k$ by $r$ whenever they differ from the corresponding entries in $\ell$.
Informally, if $\mx A$ is the adjacency matrix of a graph, the red entries (``errors'') in a contraction of rows $k$ and $\ell$ record where the graph neighbourhoods of the vertices $k$ and $\ell$ differ.

A sequence of matrices $\baa=\baa_n, \ldots, \baa_1$ is a \textit{contraction sequence} of the matrix $\baa$, whenever $\baa_1$ is $(1 \times 1)$ matrix and for all $1 \le i < n$, the matrix $\baa_i$ is a contraction of the matrix $\baa_{i+1}$.
We call such sequence a \textit{$d$-contraction sequence} if the red number of any matrix contained in it is at most $d$.
A contraction sequence is \textit{symmetric} if every contraction of a pair of rows (resp. columns) is immediately followed by a contraction of the corresponding pair of columns (resp. rows).

The \textit{twin-width} of a matrix $\baa$ is the minimum integer $d$, such that there exists a $d$-contraction sequence of $\baa$.
The \textit{symmetric twin-width} of a matrix $\baa$ is defined analogously, requiring that the contraction sequence is symmetric, and we only count the red number after both symmetric row and column contractions are performed.
See Figure~\ref{figure:tww}.
The {\em twin-width of a graph $G$} is then the symmetric twin-width of its adjacency matrix $\mx A(G)$.%
\footnote{Note that one can also define the ``natural'' twin-width of graphs which, informally, ignores the red entries on the main diagonal (as there are no loops in a simple graph).
 The natural twin-width is never larger, but possibly by one lower, than the symmetric matrix twin-width. For instance, for the sequence in Figure~\ref{figure:tww}, the natural twin-width would be at most~$2$.}

We call an ordering $\preceq$ of the rows of a matrix $\baa$ a \textit{$k$-twin-ordering} if there is a symmetric $k$-contraction sequence of $\baa$ which contracts only rows consecutive in $\preceq$.

\subsection{FO logic and transductions}
\vspace*{-1ex}

A \textit{relational signature} $\Sigma$ is a finite set of relational symbols $R_i$, each with associated arity $r_i$. A \textit{relational structure} $\baa$ with signature $\Sigma$ (or shortly a $\Sigma$-structure) is defined by a \textit{domain} $A$ and relations $R_i(\baa) \subseteq A^{r_i}$ for each relational symbol $R_i \in \Sigma$ (the relations \textit{interpret} the relational symbols). For example, graphs can be viewed as relational structures with the set of vertices as the domain and a single relational symbol $E$ with arity 2 in the relational signature.

Let $\Sigma$ and $\Gamma$ be relational signatures. An \textit{interpretation} $I$ of $\Gamma$-structures in $\Sigma$-structures is a function from $\Sigma$-structures to $\Gamma$-structures defined by a formula $\varphi_0(x)$ and a formula $\varphi_R(x_1,\mathellipsis,x_k)$ for each relational symbol $R \in \Gamma$ with arity $k$ (these formulae may use the relational symbols of $\Sigma$).
Given a $\Sigma$-structure $\baa$, $I(\baa)$ is a $\Gamma$-structure whose domain $B$ contains all elements $a \in A$ such that $\varphi_0(a)$ holds in $\baa$, and in which every relational symbol $R \in \Gamma$ of arity $k$ is interpreted as the set of tuples $(a_1,\mathellipsis,a_k) \in B^k$ satisfying $\varphi_R(a_1,\mathellipsis,a_k)$ in $\baa$.

A transduction $T$ from $\Sigma$-structures to $\Gamma$-structures is defined by an interpretation $I_T$ of $\Gamma$-structures in $\Sigma^+$-structures where $\Sigma^+$ is $\Sigma$ extended by a finite number of unary relational symbols (called \textit{marks}). Given a $\Sigma$-structure $\baa$, the transduction $T(\baa)$ is a set of all $\Gamma$-structures $\bbb$ such that $\bbb = I_T(\baa')$ where $\baa'$ is $\baa$ with arbitrary elements of $A$ marked by the unary marks.
If $\mcc$ is a class of $\Sigma$-structures, then we define $T(\mcc) = \bigcup_{\baa \in \mcc} T(\baa)$. A class $\mathcal{D}$ of $\Gamma$-structures is a \textit{transduction} of $\mcc$ if there exists a transduction $T$ such that $\mathcal{D} \subseteq T(\mcc)$.

Transductions are sometimes defined more generally, with allowed copying of the domain set. For simplicity, we define only the non-copying variant which is sufficient for our use case.

\section{Generalizing Proper $k$-Thin Graphs}\label{sec:def}

So-called $k$-thin graphs (as defined below) have been proposed and studied as a generalization of interval graphs by Mannino et al.~\cite{DBLP:journals/orl/ManninoORC07}.
Likewise, proper interval graphs have been naturally generalized into proper $k$-thin graphs~\cite{DBLP:journals/dam/BonomoE19}.

As forwarded in the introduction, we further generalize these classes into the classes of (proper) $k$-mixed-thin graphs. Note that since general $k$-mixed thin graphs have unbounded twin-width (as a generalization of interval graphs), we study only the proper case in this paper.

\begin{figure}[tb]
\vspace*{-1ex}%
        \centering\includegraphics[page=1,width=\textwidth]{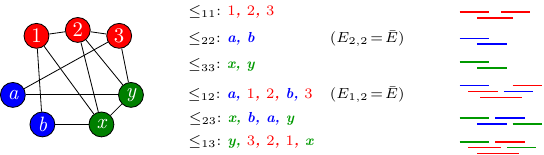}
\vspace*{-4ex}%
        \caption{An illustration of Definition~\ref{def:mixThin}.
        \textbf{Left}: a proper $3$-mixed-thin graph $G$, with the vertex set partitioned into $V_1 = \{1,2,3\}$, $V_2 = \{a,b\}$ and $V_3 = \{x,y\}$. 
        \textbf{Middle}: the six linear orders $\leq_{ij}$, and the sets $E_{i,j}$ defaulting to $E_{i,j} = E(G)$, except for $E_{1,2}$ and $E_{2,2}$.
        \textbf{Right}: a ``geometric'' proper interval representation of the orders $\leq_{ij}$ (notice -- separately for each pair~$i,j$), such that the edges between $V_i$ and $V_j$ belonging to $E_{i,j}$ are represented by intersections between intervals of colour $i$ and colour~$j$.}
        \label{figure:def-mixed-thin}
    \end{figure}
        
\begin{definition}[Mixed-thin and Proper mixed-thin]\label{def:mixThin}\normalfont 
    Let $G=(V, E)$ be a graph and $k>0$ an integer. Let $\bar E={V\choose2}\setminus E$ be the complement of its edge set.
    For two linear orders $\leq$ and $\leq'$ on the same set, we say that $\leq$ and $\leq'$ are {\em aligned} if they are the same or one is the inverse of the other.

    The graph $G$ is \textit{proper $k$-mixed-thin} if there exists a partition $\ca V=(V_1, \ldots, V_k)$ of~$V$,
    and for each $1 \le i \le j \le k$ a linear order $\le_{ij}$ on $V_i \cup V_j$ and a choice of $E_{i,j}\in\{E,\bar E\}$ (see Figure~\ref{figure:def-mixed-thin}), such that, again for every $1 \le i \le j \le k$,
    \begin{itemize}
\vspace*{-1ex}%
	\item[(a)] the restriction of $\le_{ij}$ to $V_i$ (resp.~to~$V_j$) is aligned with $\le_{ii}$ (resp.~$\le_{jj}$), and
	\item[(b)] for every triple $u,v,w$ such that ($\{u,v\}\subseteq V_i$ and $w\in V_j$) or ($\{u,v\}\subseteq V_j$ and $w\in V_i$),
	we have that if $u \lneq_{ij} v \lneq_{ij} w$ and $uw\in E_{i,j}$, then $vw\in E_{i,j}$.
	\item[(c)] for every triple $u,v,w$ such that ($\{v,w\}\subseteq V_i$ and $u\in V_j$) or ($\{v,w\}\subseteq V_j$ and $u\in V_i$),
	we have that if $u \lneq_{ij} v \lneq_{ij} w$ and $uw\in E_{i,j}$, then $uv\in E_{i,j}$.
    \end{itemize}
    General (not proper) $k$-mixed-thin graphs do not have to satisfy (c).    
    A (proper) $k$-mixed-thin graph $G$ is \textit{inversion-free} if, above, (a) is~replaced~with
    \begin{itemize}
\vspace*{-1ex}%
	\item[(a')] the restriction of $\le_{ij}$ to $V_i$ (resp.~to~$V_j$) is equal to $\le_{ii}$ (resp.~$\le_{jj}$).
    \end{itemize}
\end{definition}

We remark that the aforementioned {\em(proper) $k$-thin graphs} are those
(proper) $k$-mixed-thin graphs for which the orders $\le_{ij}$ (for $1 \le i \le j \le k$) in the definition can be chosen as the restrictions of the same linear order on~$V$,
and all~$E_{i,j}=E$ (`inversion-free' is insignificant in such case).

The class of $k$-mixed-thin graphs is thus a superclass of the class of $k$-thin graphs, and the same holds in the `proper' case.
On the other hand, the class of interval graphs is $1$-thin, but it is not proper $k$-mixed-thin for any finite $k$; the latter follows, e.g., easily from further Theorem~\ref{thm:propermix-to-tww}.

Bonomo and de Estrada~\cite[Theorem~2]{DBLP:journals/dam/BonomoE19} showed that given a (proper) $k$-thin graph $G$ and a suitable ordering $\leq$ of $V(G)$, a partition of $V(G)$ into $k$ parts compatible with $\leq$ can be found in polynomial time. 
On the other hand \cite[Theorem~5]{DBLP:journals/dam/BonomoE19}, given a partition $\ca V$ of $V(G)$ into $k$ parts, the problem of deciding whether there is an ordering of $V(G)$ compatible with $\ca V$ is \textsc{NP}-complete (again in the general and also in the proper sense).
These results do not answer whether the recognition of (proper) $k$-thin graphs is efficient or not, and neither can we at this stage say whether the recognition of (proper) $k$-mixed-thin graphs is efficient.

\subsection{Comparing (proper) $k$-mixed-thin to other classes}

We illustrate the use of our Definition~\ref{def:mixThin} by comparing it to ordinary thinness on some natural graph classes.
Recall that the (square) {\em$(r\times r)$-grid} is the Car\-tesian product of two paths of length $r$. Denote by $\overline{tK_2}$ the {\em complement of the matching} with $t$ edges.
We show several classes with unbounded thinness and bounded proper mixed-thinness.

\begin{proposition2rep}[Mannino et al.~\cite{DBLP:journals/orl/ManninoORC07}, Bonomo and de Estrada~\cite{DBLP:journals/dam/BonomoE19}]
	\label{prop:oldthin}
\\a) For every $t \ge 1$, the graph $\overline{tK_2}$ is $t$-thin but not $(t-1)$-thin.
\\b) The $(r \times r)$-grid has thinness linear in~$r$.
\\c) The thinness of the complete $m$-ary tree ($m>1$) is linear in its height.
\end{proposition2rep}
\noindent For an illustration, we briefly sketch proofs of these claims based on \cite{DBLP:journals/orl/ManninoORC07} and \cite{DBLP:journals/dam/BonomoE19}.
\begin{proof}
\textbf{a)} The graph $\overline{tK_2}$ contains exactly $t$ disjoint ``non-edges''.
It is evident that $\overline{tK_2}$ is $t$-thin since the partition can be formed by these non-edges.
On the other hand, if we had a partition with $t-1$ classes and order $\leq$ according to Definition~\ref{def:mixThin},
we could find a non-edge $xy$ such that neither of $x,y$ is the first one (in~$\leq$) in its class.
Up to symmetry, $x\lneq y$, and so there is a part $V_0\ni x,z$ where $z\lneq x$.
However, $zy$ is an edge but $xy$ is not, a contradiction to Definition~\ref{def:mixThin}(b).
\smallskip

\textbf{b)} Assume the grid is $s$-thin, and let $\leq$ be the order witnessing it.
Let $T$ be the set of the last $r^2/2$ vertices in $\leq$, and notice that $T$ has at least $\Omega(r)$ neighbours outside of $T$, denoted by $N(T)$.
Hence one of the $s$ parts, say $V_0$, in the partition of the grid satisfies $|V_0\cap N(T)|\geq\Omega(r/s)$.
Let $w\in V_0\cap N(T)$ be the least one in~$\leq$, and $t\in T$ be such that $wt$ is an edge of the grid.
Observe that every vertex $w'\in V_0\cap N(T)$ satisfies $w \le w'\leq t$, and so $w't$ is an edge by Definition~\ref{def:mixThin}(b).
However, the degree of $t$ is at most $4$, and hence $\Omega(r/s)\in\ca O(1)$ which implies~$s\geq\Omega(r)$.
\smallskip

\textbf{c)} is similar to b).
\qed
\end{proof}

\begin{proposition2rep}\label{prop:easymix}
a) For every $t \ge 1$, $\overline{tK_2}$ is inversion-free proper $1$-mixed-thin.
\\b) For all $m,n$ the $(m \times n)$-grid is inversion-free proper $3$-mixed-thin.
\\c) Every tree $T$ is inversion-free proper $3$-mixed-thin.
\end{proposition2rep}

\begin{proof}
\textbf{a)} The matching of $t$ edges, $tK_2$, is a proper interval graph, and so proper $1$-mixed-thin.
    Definition~\ref{def:mixThin} is then closed under the complement of the graph.

\smallskip
\textbf{b)}
    Let $V=\{(a,b) : 1 \le a \le m, 1 \le b \le n\}$ be the vertex set of the grid, and let $E=\{(a,b)(a + 1,b), (a,b)(a, b + 1) : 1 \le a < m, 1 \le b < n\}$ be its edge set.
    For $i=1,2,3$, let $V_i=\{(a,b) \in V : a \equiv i \mod{3}\}$ be our vertex partition (informally, into ``rows modulo~$3$'', see Figure~\ref{figure:grid-tree}, left).

    Now, for every $1 \le i \le j \le 3$, the set $V_i\cup V_{j}$ induces connected components where each one consists of only two (or one if $i=j$) rows.
    This fact allows us to naturally order, within $\leq_{ij}$, the components ``from left to right'' without the influence of the other components, and thus easily satisfy the conditions of Definition~\ref{def:mixThin}.
    Formally, for all $(a,b), (c,d) \in V_i \cup V_{j}$, let $(a,b) \le_{ij} (c,d)$ iff
\vspace*{-1ex} 
	$$a < c-1 \>\lor\> \big( |a - c| \le 1 \land b < d \big) \>\lor\> \big( a \in \{c-1,c\} \land b = d \big).$$

\textbf{c)} This case is similar to b) with BFS layering of $T$ in place of grid rows, see Figure~\ref{figure:grid-tree}, right.
    Let $T=(V,E)$, and $d(u,v)$ denote the distance between $u$ and $v$ in~$T$.
    We choose any root $r\in V(T)$ and define our partition $(V_1, V_2, V_3)$ of $V$ as the ``BFS layers from $r$ modulo~$3$'', formally as $V_i=\{v \in V : d(r, v) \equiv i \mod{3}\}$ for $i=1,2,3$.
    Observe that, for any $1 \le i < j \le 3$, the subgraph induced on $V_i\cup V_j$ consists of disjoint stars, and an order placing the stars one after another would work in Definition~\ref{def:mixThin}.

    We define $\leq_{ij}$ as follows. Let $\preceq$ be the pre-order of a breadth-first search starting in $r$.
    For any $1 \le i \le j \le 3$ and all $u, v \in V$, let $u \le_{ij} v$ iff $d(u,r) < d(v, r)-1$ or $(|d(u,r) - d(v, r)| \le 1 \>\land\> u \preceq v)$.
    It is easy to see that these orderings satisfy the requirements of Definition~\ref{def:mixThin}.
\qed
\end{proof}

\begin{figure}[tb]
\vspace*{-1ex}%
        \centering\includegraphics[page=1, scale=0.7]{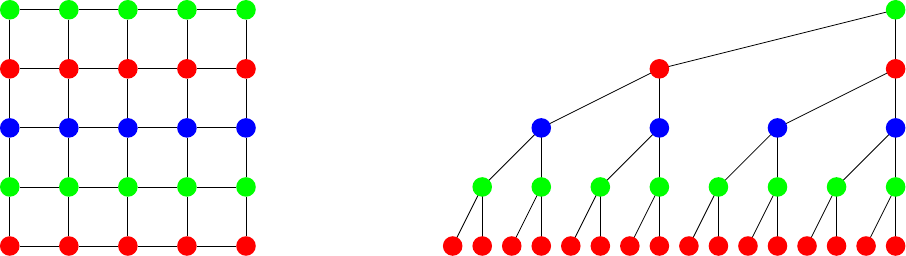}
\vspace*{-1ex}%
        \caption{\textbf{Left:} Partition of the rows of a gird into three parts in the proof of Proposition~\ref{prop:easymix}, case b). \textbf{Right:} Partition of the BFS layers of a tree into three parts in the proof of Proposition~\ref{prop:easymix}, case c).}
\vspace*{-2ex}%
        \label{figure:grid-tree}
    \end{figure}

Now we state Theorem~\ref{thm:multidimgrid}, which extends Proposition~\ref{prop:easymix}(b) to higher dimensional grids. 
The vertex sets of these grids can be viewed as the grid points in a multidimensional space, i.e., $V=\{1,2,\ldots,m\}^d$ for some $m$ and $d$.
Instead of $(a_1, \ldots, a_d) \in V$, we will write shortly $\vec a \in V$.

A {\em$d$-dimensional grid} is the Cartesian product of $d\geq1$ paths, i.e., there is an edge between $\vec a$ and $\vec b$ if $\sum_{i=1}^d |a_i-b_i| = 1$.
Similarly, the {\em$d$-dimensional full grid} is the strong product of $d\geq1$ paths, i.e., there is an edge between $\vec a$ and $\vec b$ if $\max_{1\leq i\leq d} |a_i-b_i| = 1$. Informally, the full grids are ``grids with all diagonals''.

We remark that Theorem~\ref{thm:multidimgrid} can actually be stated not only for these two types of grids but also for suitable other types of multidimensional grids
which are subgraphs of the full grid and contain the spanning ordinary grid (informally, for those which are ``something between'' ordinary and full grids),
and the proof would work the same way.

\begin{theorem2rep}\label{thm:multidimgrid}
    Let $d \ge 1$ be an arbitrary integer. Both $d$-dimensional grids and $d$-dimen\-sional full grids are inversion-free proper $3^{d-1}$-mixed-thin.
\end{theorem2rep}
\begin{proof}
The idea of the proof is the following: instead of partitioning the vertices by rows, as in Proposition~\ref{prop:easymix}(b), 
we partition them by lines parallel with the vector $(0,\ldots,0,1)$, and we count modulo~$3$ in the first $d-1$ coordinates (that is why we get $3^{d-1}$ parts). 
The proof becomes quite technical but the high-level argument stays the same; any two parts $V_i$ and $V_j$ induce multiple components which do not interleave in $\le_{ij}$. 
Again, each of these components is formed only by two lines, which allows us to order the vertices inside each of them.

 Since Definition~\ref{def:mixThin} is monotone under taking induced subgraphs, we may assume that the vertex set of the graph (the grid) is $V=\{1,2,\ldots,m\}^d$. Recall that instead of $(a_1, \ldots, a_d) \in V$, we will write simply $\vec a \in V$. 

    Let $\text{dist} : V^2 \to \mathbb{N}$ be any metric on $V$ such that for all $\vec a,\vec b\in V$, we have that
    $\max_{1\leq i\leq d} |a_i-b_i| \le \text{dist}(\vec a, \vec b)  \le \sum\limits_{i=1}^d |a_i-b_i|$
    and $\text{dist}(\vec a, \vec b) \ge \text{dist}\big(\vec a, (b_1, \ldots, b_{i-1},\\ a_i, b_{i+1}, \ldots, b_d)\big)$ for all $1 \le i \le d$.
    Let $E=\{ \vec a\vec b : \vec a, \vec b \in V, \text{dist}(\vec a, \vec b) = 1\}$ be the edge set of the grid graph.

\medskip
    Denote by $V_{i_1,\ldots,i_{d-1}} = \big\{ \vec a \in V : \bigwedge_{j=1}^{d-1} a_j \equiv (i_j-1) \!\mod{3} \big\}$.
    We will use $\big(V_i: i \in \{1,2,3\}^{d-1}\big)$ as our vertex partition.
    For all tuples $i, j \in \{1,2,3\}^{d-1}$ and all $\vec a,\vec b \in V_i \cup V_j$,
    we declare the predicate $\text{nearby}(\vec a, \vec b) =\left(\bigwedge_{k=1}^{d-1} |a_k - b_k| \le 1\right)$.
    Now, let $\vec a \le_{ij} \vec b$ if and~only~if
    \begin{align*}
	\Bigg( \neg &\text{nearby}(\vec a, \vec b) \land \left. \bigvee\limits_{k=1}^{d-1} \left( a_k + 1 < b_k \land \bigwedge\limits_{\ell=1}^{k-1} |a_\ell - b_\ell| \le 1 \right) \right)
    \\
    	&\lor\>\vec a =\vec b \>\lor\> \Big( \text{nearby}(\vec a, \vec b) \land a_d < b_d \Big)
    \\
    	&\lor \left( \text{nearby}(\vec a, \vec b) \land a_d = b_d \land \bigvee\limits_{k=1}^{d-1} \left( a_k < b_k \land \bigwedge\limits_{\ell=1}^{k-1} a_\ell = b_\ell \right) \right)
    .\end{align*}

    Let $i, j \in \{1,2,3\}^{d-1}$,
    $\vec a, \vec b \in V_i$ and $\vec c \in V_j$ be such that $\vec a\vec c \in E$,
    and either $\vec a \lneq_{ij} \vec b \lneq_{ij} \vec c$ or $\vec c \lneq_{ij} \vec b \lneq_{ij} \vec a$.
    
    \smallskip
    From $\max_{1\leq i\leq d} |a_i-c_i| \le \text{dist}(\vec a, \vec c) = 1$ we get that $\text{nearby}(\vec a, \vec c)$ is true,~and~so 
    both $\text{nearby}(\vec a, \vec b)$ and $\text{nearby}(\vec b, \vec c)$ are true. 
    Hence $(a_1, \ldots, a_{d-1}) = (b_1, \ldots, b_{d-1})$, and either $a_d < b_d \le c_d \le a_d + 1$ or $a_d - 1 \le c_d \le b_d < a_d$. 
    Then $\vec a\vec b, \vec b\vec c \in E$ follows from the properties of our metric $\text{dist}(\cdot)$.
\qed
\end{proof}

To further illustrate the strength of the new concept, we show that proper $k$-mixed-thin graphs generalize the following class~\cite{Jedelsky2021thesis},\todo{would be nice to relate to the BP}
which itself can be viewed as a natural generalization of proper interval graphs and $k$-fold proper interval graphs (a subclass of interval graphs whose representation can be decomposed into $k$ proper interval subrepresentations):

\begin{theorem2rep}\label{thm:pathsInGraphs}
    Let $G=(V, E)$ be a proper intersection graph of (vertex sets of) paths in some subdivision of a fixed connected graph $H$ with $m$ edges, and let $k$ be the sum of the number of paths and the number of cycles in $H$.
    Then $G$ is a proper $(m^2k)$-mixed-thin graph.
\end{theorem2rep}
\begin{proof}
\begin{figure}[t]
        \centering\includegraphics[page=1, width=\textwidth]{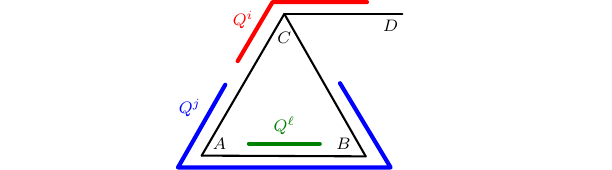}
        \caption{An illustration of the proof of Theorem~\ref{thm:pathsInGraphs}. The graph $H$ is in black. There are three parts of $\ca Q$, and their underlying graphs are drawn in red, green, and blue. $Q^i$ is a path with both edges being end-edges, $Q^j$ is a cycle with end-edges $AC$ and $BC$, and $Q^\ell$ is a single end-edge. Finding the ordering $\le_{j\ell}$ is case~\ref{thm5:all}, finding $\le_{i\ell}$ is case~\ref{thm5:no}, and finding $\le_{ij}$ is case~\ref{thm5:path} (the cases are analyzed in the text and in Figure~\ref{figure:flow}).}
        \label{figure:simple_thm5}
    \end{figure}
    Consider a (now fixed) subdivision $H'$ of $H$ such that each vertex $v\in V$ has its representative path $P_v\subseteq H'$.
    We may assume (by possibly taking a finer subdivision) that no end of $P_v$ (in~$H'$) is a vertex in $H$, i.e., that the ends of all $P_v$ ``lie inside'' the subdivided edges of~$H$.

    For each vertex $v \in V$, let $S_v\subseteq H$ be the unique minimal subgraph such that its corresponding subdivision $S_v'\subseteq H'$ contains~$P_v$.
    Observe there are at most two edges in $S_v$ whose subdivisions contain the ends of~$P_v$, and the remainder of $S_v$ is a path (or empty). We call these at most two edges the \textit{end-edges} (of $v$), and we denote the set containing them $E_v$.
    
    Denote by $Q_v := (S_v, E_v)$. We shall view $Q_v$ as a graph with one or two marked edges, and so by vertices (edges) of $Q_v$, we will mean vertices (edges) of $S_v$, and by end-edges of $Q_v$, we will mean elements of $E_v$.
    
    Consider the partition $\ca Q:=(V[Q_u]:{u \in V})$ of $V$, where $V[Q]=\{v: Q_v=Q\}$.
    For a part $V_i\in \ca Q$, if $V_i = V[Q]$, then we call $Q$ the \textit{underlying graph} of $V_i$, and we denote it $Q^i$.
    Notice that $|\ca Q| \le m^2k$ (since each underlying graph is either a path or a cycle). For an illustration, see Figure~\ref{figure:simple_thm5}.
\smallskip

    For each part $V_i \in \ca Q$, we choose $\le_{ii}$ to be the linear order of $V_i$ determined by in-order enumeration of the ends of representatives of $V_i$ on either of the end-edges of $Q^i$
    (this is unambiguous and sound since we have a proper representation, and since we allow inversions in Definition~\ref{def:mixThin}(a)\,).

    Now, for each pair $V_i$ and $V_j$ of parts in~$\ca Q$, we give a linear order and a choice of $E_{i,j}$ satisfying Definition~\ref{def:mixThin} (with respect to~$\ca Q$). There are six cases which need to be considered, depending on the relation between $Q^i$ and $Q^j$, see the diagram in Figure~\ref{figure:flow}.

\begin{figure}[t]
        \centering\includegraphics[page=1, width=0.95\textwidth]{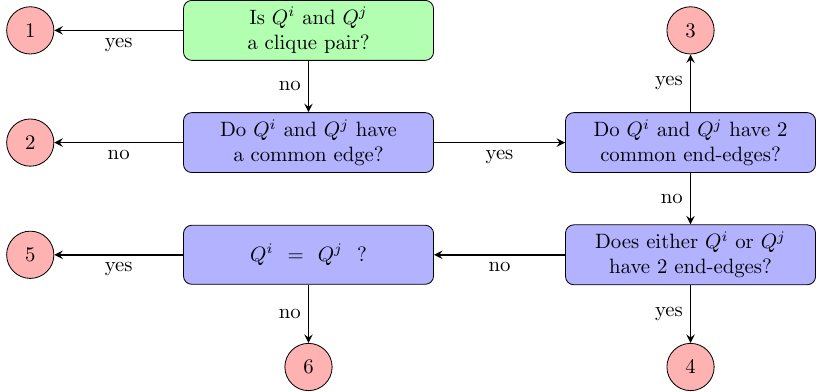}
\vspace*{-1ex}%
        \caption{A flowchart depicting the case analysis in the final part of the proof of Theorem~\ref{thm:pathsInGraphs}, starting in the green box. The numbers refer to cases analyzed in the text.}
\vspace*{-2ex}%
        \label{figure:flow}
    \end{figure}

    \begin{enumerate}
    \item\label{thm5:all} We say that the pair $Q^i, Q^j$ is a \textit{clique pair} if for every $u \in V_i, v \in V_j$, $uv \in E$.
    If $Q^i, Q^j$ is a clique pair, then we set $E_{i,j}=E$ and any ordering respecting the orders on $V_i$ and $V_j$ is fine. Note that in all the following cases, all edges shared by $Q^i$ and $Q^j$ are end-edges in both of them (otherwise all representative paths would intersect on this edge).
    \item\label{thm5:no} If $Q^i$ and $Q^j$ do not share any edge (and are not a clique pair), then there are no edges between $V_i$ and $V_j$ in $G$, and the same choice as in the case~\ref{thm5:all} is valid.
    \item\label{thm5:2} If $Q^i$ and $Q^j$ have two common end-edges (and are not a clique pair), then $Q^i$ and $Q^j$ are two paths, and their union is a cycle in $H$. This case is a little more complex and we deal with it later.
            Note that in all the following cases, $Q^i$ and $Q^j$ share exactly one end-edge.
     \item\label{thm5:path} Suppose that $Q^i$ has two end edges, one of which, $e$, is also an end-edge of $Q^j$, and the other one is not. Also suppose they are not a clique pair.
     Observe that $Q^j$ either has only one edge (namely $e$) or it has two end-edges (it cannot be a cycle with one end-edge because then $Q^i$ and $Q^j$ would be a clique pair).
      In both cases, $G[V_i \cup V_j]$ is a proper interval graph, we choose the ordering $\le_{ij}$ by enumerating the endpoints along $e$ (in one of the two possible directions), and we set $E_{i,j}=E$.
     \item\label{thm5:same} If $Q^i = Q^j$, then $i = j$, and we have already chosen the ordering $\le_{ii}$. We set $E_{i,i}=E$, and this is valid since $G[V_i]$ is a proper interval graph.
     \item\label{thm5:last} Suppose that both $Q^i$ and $Q^j$ have only one end-edge $e$, shared by both of them, but $Q^i \ne Q^j$. Also suppose they are not a clique pair. This implies that $e$ is the only edge in, say, $Q^i$, and that $Q^j$ is a cycle in $H$. We deal with this case in the final part of the proof, together with case~\ref{thm5:2}.
    \end{enumerate}
    
    Now we will finish the proof by dealing with cases~\ref{thm5:2} and~\ref{thm5:last}.
    Observe that case~\ref{thm5:2} can be reduced to case~\ref{thm5:last} by considering one of the paths, say, $Q^i$ to be a single edge.
    This way we obtain a subcase of case~\ref{thm5:last} in which $V_i$ is a clique.
    Thus it is enough to solve case~\ref{thm5:last}.
    
    Recall that $e$ is the common end-edge of $Q^i$ and $Q^j$, and that it is the only edge of $Q^i$.
    Let $e^+$ and $e^-$ be the endpoints of $e$, and let $e'$ be the subdivision of $e$ in $H'$.
    For $v \in V_j$, let $P_v^+$ (resp. $P_v^-$) be the maximal path in $P_v \cap e'$ containing $e^+$ (resp. $e^-$), see Figure~\ref{figure:final}.
    For $u,v \in V_j$, let $u \le_{jj}^+ v$ (resp $u \le_{jj}^- v$) iff $P_u^+ \subseteq P_v^+$ (resp. $P_u^- \subseteq P_v^-$).
    Observe that $\le_{jj}^+$ and $\le_{jj}^-$ are linear orders aligned with $\le_{jj}$ (one of them is $\le_{jj}$ and the other one is the inverse thereof).
    We may suppose that for $u,v \in V_i$, $u \le_{ii} v$ iff the left (resp. right) endpoint of $P_u$ is closer in $e'$ to $e^-$ than the left (resp. right) endpoint of $P_v$ (otherwise we would invert $\le_{ii}$).

    \begin{figure}[t]
\hspace*{-1.5ex}%
        \centering\includegraphics[page=1, width=1.03\textwidth]{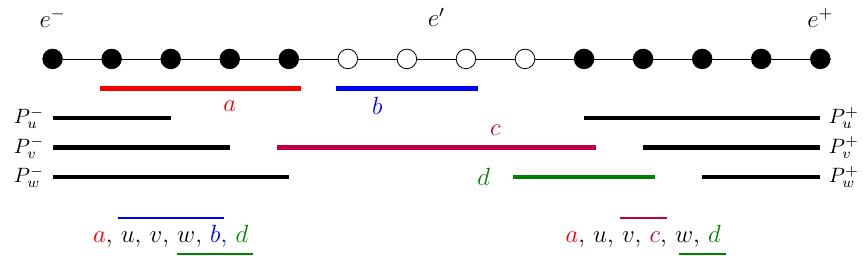}
        \caption{An illustration of the final part of the proof of Theorem~\ref{thm:pathsInGraphs} (cases~\ref{thm5:2} and~\ref{thm5:last}). \textbf{Top:} the subdivided edge $e'$. Black vertices belong to $P_x$ for some $x \in V_j$. \textbf{Middle:} Representatives of vertices from $V_i$ and $V_j$ (only parts intersecting $e'$). Either $b$ or $c$ is present, not both. \textbf{Bottom:} the ordering $\le_{ij}$, one with $b$ and one with $c$. The thin lines represent maximal bicliques in $E_{i,j}$ (since $E_{i,j}=\bar E$).}
        \label{figure:final}
    \end{figure}
    
    \smallskip
    Now we finally define $\le_{ij}$ and $E_{i,j}$. Restricted to $V_i$ (resp. $V_j$), $\le_{ij}$ is the same as $\le_{ii}$ (resp. $\le_{jj}^-$).
     For $u \in V_j$ and $v \in V_i$, $u \le_{ij} v$ if and only if $P_u^- \cap P_v = \emptyset$.
     In contrast with the other cases, we choose $E_{i,j} = \bar E$.
    
    Let us now check that our choice satisfies Definition~\ref{def:mixThin}. Let $u,v,w\in V_i\cup V_j$ such that $u \le_{ij} v \le_{ij} w$ and $uw \notin E$.
    If $u \in V_i$ and $w \in V_j$, then $P_u \cap P_w^- \ne \emptyset$ by definition of $\le_{ij}$, which would be a contradiction. Thus suppose $u \in V_j$ and $w \in V_i$.
    
    Suppose $v \in V_j$. Since $v \le_{ij} w$, we get $P_v^- \cap P_w = \emptyset$. Since $P_u^+ \cap P_w = \emptyset$ and $P_v^+ \subseteq P_u^+$ (since $v \le_{jj}^+ u$), we get $P_v^+ \cap P_w = \emptyset$. Thus $vw \notin E$.
    
    Lastly suppose $v \in V_i$. Since $u \le_{ij} v$, we get $P_u^- \cap P_v = \emptyset$.
    Since the right endpoint of $P_v$ is to the left of the right endpoint of $P_w$ and $P_u^+ \cap P_w = \emptyset$, we get $P_u^+ \cap P_v = \emptyset$.
    Thus $uv \notin E$, which concludes the proof. \qed

\end{proof}

\section{Proper $k$-Mixed-thin Graphs Have Bounded Twin-width}\label{sec:tww}

In the founding series of papers, Bonnet et al.~\cite{DBLP:conf/focs/Bonnet0TW20,DBLP:conf/soda/BonnetGKTW21,DBLP:conf/icalp/BonnetG0TW21,DBLP:conf/stoc/BonnetGMSTT22}
proved that many common graph classes (in addition to aforementioned posets of bounded width) are of bounded twin-width.
Their proof methods have usually been indirect (using other technical tools such as `mixed minors'), but for a few classes including proper interval graphs and multidimensional grids and full grids (cf.~Theorem~\ref{thm:multidimgrid}) they provided a direct construction of a contraction sequence.

We have shown \cite{DBLP:conf/iwpec/BalabanH21} that a direct and efficient construction of a contraction sequence is possible also for posets of width~$k$.
Stepping further in this direction, our proper $k$-mixed-thin graphs, which largely generalize proper interval graphs, still have bounded twin-width, as we are now going to show with a direct and efficient construction of a contraction sequence for them.


Before stating the result, we mention that $1$-thin graphs coincide with interval graphs which have unbounded twin-width by \cite{DBLP:conf/soda/BonnetGKTW21},
and hence the assumption of `proper' in the coming statement is necessary.

\begin{theorem}\label{thm:propermix-to-tww}
    Let $G$ be a proper $k$-mixed-thin graph. Then the twin-width of~$G$, i.e., the symmetric twin-width of $\mx{A}(G)$, is at most $9k$.
    The corresponding contraction sequence for $G$ can be computed in polynomial time from the vertex partition $(V_1,\ldots,V_k)$ and the orders $\leq_{ij}$ for $G$ from Definition~\ref{def:mixThin}.
\end{theorem}

Referring to Definition~\ref{def:mixThin}, the proper $k$-mixed-thin graph $G$ is associated with a vertex partition $(V_1,\ldots,V_k)$ and linear orders $\leq_{ij}$.
In the course of proving Theorem~\ref{thm:propermix-to-tww}, an adjacency matrix $\mx{A}(G)$ of $G$ is always obtained by ordering the $k$ parts arbitrarily, and then inside each part using the order $\le_{ii}$. 
Furthermore, we denote $\mx{A}_{i,j}(G)$ the submatrix with rows from $V_i$ and columns from $V_j$.

We would like to talk about parts (``areas'') of a~$(p \times q)$ matrix $\mx M$. To do so, we embed such a matrix into the plane as a $((p+1) \times (q+1))$-grid, where entries of the matrix are represented by labels of the bounded square faces of the grid. 
We call a \emph{boundary} any path in the grid, which is also a separator of the grid.
In this view, we say that a matrix entry $a$ is {\em next to} a boundary if at least one of the vertices of the face of $a$ lies on the boundary.

Note that the grid has four corner vertices of degree~$2$, and a \emph{diagonal boundary} is a shortest (i.e., geodesic) path going either between the top-left and the bottom-right corners, or between the top-right and the bottom-left corners.
We say that two diagonal boundaries $b_1$ and $b_2$ are \textit{crossing} if $b_1$ contains two grid vertices $v$ and $v'$ not contained in $b_2$, such that $v$ and $v'$ belong to different parts of the matrix separated by $b_2$.
We call a matrix $\mx M$ {\em diagonally trisected} if $\mx M$ contains two non-crossing diagonal boundaries with the same ends which separate the matrix into three {\em parts}.
The part bounded by both diagonal boundaries is called the {\em middle part}.
See Figure~\ref{figure:diagonalBoundary}.

\begin{figure}[tb]
    \centering\includegraphics[page=1]{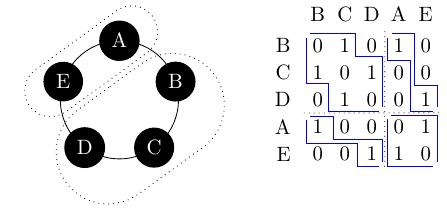}
    \vspace*{-1ex}%
    \caption{On the left, there is a partition of $C_5$ into parts $(\{B,C,D\},\{A,E\})$, which with the ordering e.g.\ $A \le B \le C \le D \le E$ certifies that $C_5$ is proper 2-thin, therefore proper 2-mixed-thin as well. 
	On the right, there is an adjacency matrix of $C_5$, together with eight blue diagonal boundaries obtained by the process described in Lemma \ref{lem:matrixboundaries}.}
    \label{figure:diagonalBoundary}
\end{figure}

Now we prove Lemma~\ref{lem:matrixboundaries}, which states that each submatrix $\mx{A}_{i,j}(G)$ is diagonally trisected in a special way. These diagonal trisections will later be crucial for finding a good symmetric contraction sequence of $\mx{A}(G)$.

\begin{lemma2rep}\label{lem:matrixboundaries}
    Let $G$ be a proper $k$-mixed-thin graph. For all $1 \le i, j \le k$, the submatrix $\mx{A}_{i,j}(G)$ is diagonally trisected,
    such that each part has either all entries~$0$ or all entries~$1$, with the exception of entries on the main diagonal~of~$\mx{A}(G)$. Furthermore, the diagonal boundaries of the submatrix $\mx{A}_{i,i}(G)$ are symmetric (w. r. to the main diagonal).
\end{lemma2rep}
\begin{proof}
Let $1 \le i, j \le k$. We may assume that $i \le j$ since the matrix $\mx{A}(G)$ is symmetric.
Recall that by Definition~\ref{def:mixThin}, there is a choice of an ordering $\le_{ij}$ and $E_{i,j} \in \{E(G), \overline{E(G)}\}$. 
    Observe that the stated property of a submatrix $\mx{A}_{i,j}(G)$ is preserved by possibly reversing the ordering of rows or columns of the submatrix, hence we can assume that the restriction of $\le_{ij}$ to $V_i$ (resp. $V_j$) is equal to $\le_{ii}$ (resp. $\le_{jj}$).
    With this assumption, we will prove that the two diagonally trisecting boundaries go from the top-left corner to the bottom-right corner of $\mx{A}_{i,j}(G)$.
    Furthermore observe that the stated property is preserved by complementing the submatrix (that is, swapping entries 1 and 0 except for the main diagonal of $\mx{A}(G)$), hence we can assume that $E_{i,j}=E(G)$.

    For all $u \in V_i$ and all $v \in V_j$, denote by $a_{uv}$ the entry of $\mx{A}(G)$ in the intersection of the row corresponding to $u$ and the column corresponding to $v$.
    Furthermore, we say that the entry $a_{uv}$ {\em belongs to the rows} whenever $u \le_{ij} v$, and that it {\em belongs to the columns} otherwise.
    
    Observe that there is a diagonal boundary between entries belonging to the rows, and entries belonging to the columns. We find one of our desired diagonal boundaries inside the area belonging to the rows, and the other inside the area belonging to the columns.

    Let $u \in V_i, v \lneq_{jj} w \in V_j \setminus\{ u \}$ be such that both $a_{u,v}$ and $a_{u,w}$ belong to rows. Then it follows from Definition \ref{def:mixThin} that if $a_{u,w}=1$, then $a_{u,v}=1$. 
Informally, this means that a matrix entry $1$ belonging to the rows is ``propagated to the left''. 
    
    Similarly, let $w \in V_j, u \lneq v \in V_i \setminus\{ w \}$ be such that both $a_{u,w}$ and $a_{v,w}$ belong to rows. Then it follows from Definition \ref{def:mixThin} that if $a_{u,w}=1$ then $a_{v,w}=1$. 
Informally, this means that a matrix entry $1$ belonging to the rows is ``propagated down''.

    Therefore there is a diagonal boundary splitting the entries belonging to the rows between those which contain ones, and those which contain zeros, with the possible exception of entries on the main diagonal of $\mx{A}(G)$.
The diagonal boundary splitting the entries belonging to the columns can be obtained analogously.

Finally, the symmetry of diagonal boundaries in the case of $i=j$ follows from the symmetry of the matrix.
\qed
\end{proof}

\begin{proof}[of Theorem \ref{thm:propermix-to-tww}]
    For each $1 \le i \neq j \le k$, by Lemma~\ref{lem:matrixboundaries}, the submatrix $\mx{A}_{i,j}(G)$ of $\mx A(G)$ is diagonally trisected such that each part has all entries equal (i.e., all $0$ or all $1$). 
    The case of $i=j$ is similar, except that the entries on the main diagonal might differ from the remaining entries in the same area.
    Furthermore, since the matrix $\mx{A}$ is symmetric, we can assume that the diagonal boundaries are symmetric as well.

    We generalize this setup to matrices with {\em red} entries $r$; these come from contractions of non-equal entries in $\mx A(G)$, cf.~Subsection~\ref{sub:twwdef}. 
    Considering a~matrix $\mx M=(m_{uv})_{u,v}$ obtained by symmetric contractions from $\mx A(G)$, we assume that
    \begin{itemize}\vspace*{-1ex}
	\item $\mx M$ is {\em consistent} with the partition $(V_1,\ldots,V_k)$, meaning that only rows and columns from the same part have been contracted in~$\mx A(G)$,
	\item $\mx M$ is {\em red-aligned}, meaning that each submatrix $\mx M_{i,j}$ obtained from $\mx{A}_{i,j}(G)$ by row contractions in $V_i$ and column contractions in $V_j$,
	is diagonally trisected such that (again with the possible {\em exception} of entries on the main diagonal of $\mx M$):
	each of the three parts has all entries either from $\{0,r\}$ or from $\{1,r\}$, and moreover, the entries $r$ are only in the middle part and next to one of the diagonal boundaries, and
    \item the diagonal boundaries of $\mx M$ are also symmetric, that is, there is a boundary between $m_{uv}$ and $m_{uw}$ iff there is a boundary between $m_{vu}$ and $m_{wu}$.
    \end{itemize}

    We are going to show that there is a symmetric matrix-contraction sequence starting from $\mx M^0:=\mx{A}(G)$ down to an $(8k \times 8k)$ matrix $\mx M^t$,
    such that all square matrices $\mx M^m$, $0\leq m\leq t$, in this sequence are consistent with $(V_1, \ldots, V_k)$, red-aligned, and have red number at most $9k$. Furthermore, the matrices in our sequence are symmetric, and so are the diagonal boundaries.
    Hence we only need to observe the red values of the rows.
    Then, once we get to $\mx M^t$, we may finish the contraction sequence arbitrarily while not exceeding the red value~of~$8k$.

    Assume we have got to a matrix $\mx M^m$, $m\geq0$, of the claimed properties in our sequence, and $\mx M^m$ has more than $8k$ rows. 
    The induction step to the next matrix $\mx M^{m+1}$ consists of two parts:
    \begin{itemize}\vspace*{-1ex}
	\item[(i)] We find a pair of consecutive rows from (some) one part of $(V_1, \ldots, V_k)$, such that their contraction does not yield more than $7k$ red entries.
	\item[(ii)] After we do this row contraction followed by the symmetric column contraction to $\mx M^{m+1}$ (which may add one red entry up to each other row of $\mx M^{m+1}$), we show that the red value of any other row does not exceed~$7k+2k=9k$.
    \end{itemize}

    Part (i) importantly uses the property of $\mx M^m$ being red-aligned, and is given separately in the next claim:
    \begin{claim2rep}\label{clm:onecontract}
	If a matrix $\mx M^m$ satisfies the above claimed properties and is of size more than $8k$,
	then there exists a pair of consecutive rows from one part in $\mx M^m$, such that their contraction gives a row with at most $7k$ red entries
	(a technical detail; this number includes the entry coming from the main diagonal of $\mx M^m$).
	After this contraction in $\mx M^m$, the newly created matrix $\mx M^{m+1}$ will be again red-aligned.
    \end{claim2rep}
    \begin{proof}[Subproof]
    Each boundary of a diagonally trisected submatrix $\mx M^m_{i,j}$, where $1 \le i, j \le k$, can be split into vertical and horizontal segments, and the length of a segment is the number of its grid edges.
    For each row $q$ of $\mx M^m$ and each boundary $b$, we define the \emph{horizontal value of $q$ restricted to $b$} as the length of the horizontal segment of $b$ between the rows $q$ and~$q+1$, or $0$ if there is no such horizontal segment.
    The \emph{horizontal value $h(q)$ of a row $q$} is then the sum of the horizontal values of $q$ restricted to each boundary $b$, over all $2k$ diagonal boundaries $b$ crossing $q$. See Figure~\ref{figure:big-matrix} for an illustration.
    
    \begin{figure}[t]
        \centering\includegraphics[page=1, width=1.02\textwidth]{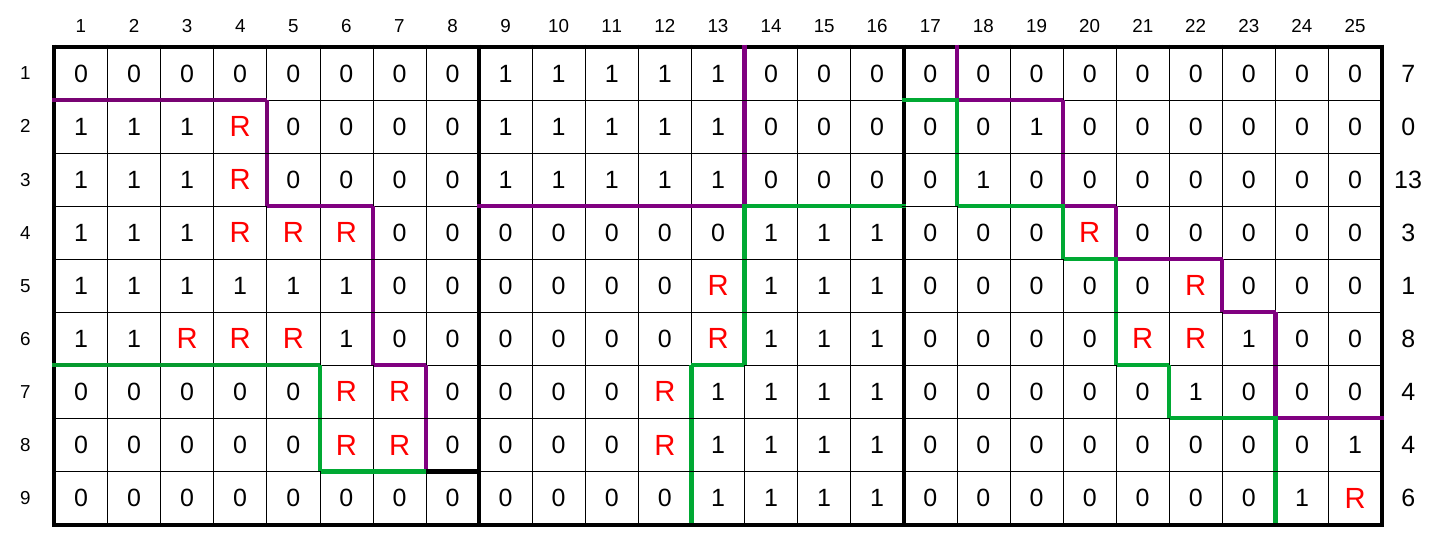}
        \caption{An illustration of Claim~\ref{clm:onecontract} for $k=3$ and $s=9$, depicting $\mx M^m_3$. The submatrices $\mx M^m_{3,1}$, $\mx M^m_{3,2}$, $\mx M^m_{3,3}$ are separated by thick lines.
The numbers left of the matrix and above it are the labels of the rows and columns. The diagonal boundaries are drawn in green and purple, except for the borders of the submatrices and the segment between rows 8 and 9 in column 8 (which is shared by both colors). The boundaries go from the top-left corner to the bottom-right corner in $\mx M^m_{3,1}$ and $\mx M^m_{3,3}$, and from the bottom-left corner to the top-right corner in $\mx M^m_{3,2}$ (this is possible by inverting the orders $\le_{ij}$ cf. Definition~\ref{def:mixThin}). 
Observe that $\mx M^m_{3,3}$ is symmetric.    
        \\ On the right of each row of the matrix, one can read the horizontal values of the rows. 
	Observe that for row 8, the shared segment is counted twice (which is a slight over-counting). In the proof, we state that there is a pair of consecutive rows with sum of horizontal values at most $\frac{4ks}{s-1} = \frac{27}{2}$. Here, all pairs except for rows 3 and 4 satisfy this condition.}
\vspace*{-1ex}%
        \label{figure:big-matrix}
    \end{figure}
    
    Observe that, since each red entry in $\mx M^m_{i,j}$ is in the middle part of the diagonal trisection of $\mx M^m_{i,j}$ (because $\mx M^m$ is red-aligned), the number of red entries in a row $q$ which are next to a particular horizontal segment of a boundary $b$ exceeds the length of this segment by at most one. This means that the total number of red entries in any row $q$ of $\mx M^m$ is at most $h(q) + 2k + 1$ (adding also one possible exceptional entry on the main diagonal).
    We can hence focus on the horizontal value only.

    Since the size of  $\mx M^m$ is more than~$8k$, the largest part of $(V_1, \ldots, V_k)$ restricted (by previous contractions) to $\mx M^m$ has size $s\geq9$.
    So, let us fix $i\in\{1,\ldots,k\}$ such that the submatrix $\mx M^m_{i,1}$ has $s$ rows, and notice that for each $j\in\{1,\ldots,k\}$ the submatrix $\mx M^m_{i,j}$ has $s$ rows and at most $s$ columns.
    Let us denote $\mx M^m_{i}$ the submatrix of $\mx M^m$ obtained by union of $\mx M^m_{i,j}$ over all $j\in\{1,\ldots,k\}$.
    The sum of all horizontal values restricted to one boundary of any $\mx M^m_{i,j}$ is, by the definition, at most 
     the number of its columns (so, at most~$s$).
    Consequently, the sum of horizontal values of all rows of $\mx M^m_i$ is at most~$2ks$.

    Furthermore, if we contract a row $q$ with the next row $q+1$, then the horizontal value of the resulting row will be at most~$h(q)+h(q+1)$.%
\footnote{This value~$h(q)+h(q+1)$ plays the role of the red potential in Section~\ref{sec:red-potential}.}
    Therefore, by the pigeon-hole principle, among the $s-1$ pairs of consecutive rows in $\mx M^m_{i}$ there is a pair whose contraction yields the horizontal value at most $\frac{2\cdot2ks}{s-1}$.
    Since $s\geq9$, we have $h(q)\leq\lfloor\frac{4ks}{s-1}\rfloor\leq\lfloor\frac{4k\cdot9}{9-1}\rfloor\leq5k-1$ for the now contracted row~$q$.
    The~number of red entries after the contraction hence is at most $h(q)+2k+1\leq7k$.

    Finally, as for the red-alignedness property after the contraction of rows $q$ and $q+1$, we observe that in every column $c$ such that the matrix entries at $(q,c)$ and $(q+1,c)$ are both {\em not} in the middle part (of the respective trisected submatrix $\mx M^m_{i,j}$),
    these entries have equal value which is not red, and this stays so after the contraction.
    The same can be said when the entries at $(q,c)$ and $(q+1,c)$ have equal non-red value in the middle part of $\mx M^m_{i,j}$.
    Otherwise, at least one of the entries at $(q,c)$ and $(q+1,c)$ is next to a diagonal boundary in $\mx M^m_{i,j}$ (or on the main diagonal of $\mx M^m$), and so we do not care that it may become red.
    After the contraction, the horizontal boundary segment is simply shifted right above or right below the contracted entry at $(q,c)$, so that this entry stays in the middle part.
    \subqed
    \end{proof}

    In part (ii) of the induction step, we fix any row $i\in\{1,\ldots,k\}$ of $\mx M^{m+1}$.
    Row $i$ initially (in $\mx M^0$) has no red entry, and it possibly got up to $7k$ red entries in the previous last contraction involving it.
    After that, row $i$ has possibly gained additional red entries only through column contractions,
    and such a contraction leading to a new red entry in row~$i$ (except on the main diagonal which has been accounted for in Claim~\ref{clm:onecontract})
    may happen only if the two non-red contracted entries lied on two sides of the same diagonal boundary.
    Since we have $2k$ such boundaries throughout our sequence, we get that the number of red entries in $\mx M^{m+1}$ is indeed at most~$7k+2k=9k$.

\smallskip
\ifx\proof\inlineproof
    We have finished the induction step, and so the whole proof of Theorem~\ref{thm:propermix-to-tww} by the above outline.
\else
    We have finished the induction step, and so the whole proof by the above outline.
\fi
    Note that all steps are efficient, including Claim~\ref{clm:onecontract} since at every step there is at most a linear number of contractions which we are choosing from.
\qed
\end{proof}

The following result is a corollary of our Theorem~\ref{thm:propermix-to-tww} and Theorem~21 from~\cite{DBLP:conf/focs/Bonnet0TW20}.

\begin{corollary}[based on \cite{DBLP:conf/focs/Bonnet0TW20}]\label{cor:propermix-FOmc}
    Assume a proper $k$-mixed-thin graph $G$, given alongside with the vertex partition and the orders from Definition~\ref{def:mixThin}.
    Then FO model checking on $G$ is solvable in FPT time with respect to~$k$.
\qed
\end{corollary}

It may be possible that the constant 9 in the statement of Theorem~\ref{thm:propermix-to-tww} could be slightly improved, by counting the red entries more carefully (which would probably make the proof more complicated).

However, the bound cannot be improved below linear dependence, as we now show in Proposition~\ref{prop:propermix-to-tww-lower}. The construction of the graph $G$ in the proof is based simply on the construction of a poset of high twin-width from \cite[Proposition~2.4]{DBLP:conf/iwpec/BalabanH21},
which is first modified by ``doubling'' its chains, and then made into a simple undirected~graph.

\begin{proposition2rep}\label{prop:propermix-to-tww-lower}
    For every integer $k\geq1$, there exists an inversion-free proper $(2k+1)$-mixed-thin graph $G$ such that the twin-width of $G$ is at least $k$.
\end{proposition2rep}

\begin{proof}
\begin{figure}[th]
$$
\begin{tikzpicture}[scale=0.75]
\tikzstyle{every node}=[draw, shape=circle, minimum size=3pt,inner sep=0pt, fill=black]
\node[label=left:$c_1^0$] at (0,0) (a0) {}; \node at (0,1) (a1) {}; \node at (0,2) (a2) {};
\node at (0,3) (a3) {}; \node at (0,4) (a4) {}; \node at (0,5) (a5) {}; 
\node at (0,6) (a6) {}; \node at (0,7) (a7) {}; \node[label=left:$c_1^8$] at (0,8) (a8) {}; 
\node[label=left:$c_2^0$] at (3,0) (b0) {}; \node at (3,1) (b1) {}; \node at (3,2) (b2) {};
\node at (3,3) (b3) {}; \node at (3,4) (b4) {}; \node at (3,5) (b5) {}; 
\node at (3,6) (b6) {}; \node at (3,7) (b7) {}; \node[label=left:$c_2^8$] at (3,8) (b8) {}; 
\node[label=left:$c_3^0$] at (6,0) (c0) {}; \node at (6,1) (c1) {}; \node at (6,2) (c2) {};
\node at (6,3) (c3) {}; \node at (6,4) (c4) {}; \node at (6,5) (c5) {}; 
\node at (6,6) (c6) {}; \node at (6,7) (c7) {}; \node[label=right:$c_3^8$] at (6,8) (c8) {}; 
\node[label=right:$c_4^0$] at (9,0) (d0) {}; \node at (9,1) (d1) {}; \node at (9,2) (d2) {};
\node at (9,3) (d3) {}; \node at (9,4) (d4) {}; \node at (9,5) (d5) {}; 
\node at (9,6) (d6) {}; \node at (9,7) (d7) {}; \node[label=right:$c_4^8$] at (9,8) (d8) {};
\draw[very thick] (a0) -- (a8) (b0) -- (b8) (c0) -- (c8) (d0) -- (d8) ; 
\draw[thick] (a0) -- (b3) (b0) -- (c3) (c0) -- (d3) ; 
\draw[thick] (a3) -- (b6) (b3) -- (c6) (c3) -- (d6) ; 
\draw[very thin] (a1) -- (b6) ;
\draw[very thin,dashed] (a0) -- (b4) (a0) -- (b5) (a0) -- (b6) (a0) -- (b7) (a0) -- (b8) ;
\draw[very thin,dashed] (a1) -- (b7) (a1) -- (b8) ;
\draw[very thin,dashed] (a2) -- (b6) (a2) -- (b7) (a2) -- (b8) ;
\draw[very thin,dashed] (a3) -- (b7) (a3) -- (b8) ;
\draw[very thin] (b1) -- (c6) ;
\draw[dotted] (b0) -- (c4) (b0) -- (c5) (b0) -- (c6) (b0) -- (c7) (b0) -- (c8) ;
\draw[dotted] (b1) -- (c7) (b1) -- (c8) ;
\draw[dotted] (b2) -- (c6) (b2) -- (c7) (b2) -- (c8) ;
\draw[dotted] (b3) -- (c7) (b3) -- (c8) ;
\draw[very thin] (c1) -- (d6) ;
\draw[dotted] (c0) -- (d4) (c0) -- (d5) (c0) -- (d6) (c0) -- (d7) (c0) -- (d8) ;
\draw[dotted] (c1) -- (d7) (c1) -- (d8) ;
\draw[dotted] (c2) -- (d6) (c2) -- (d7) (c2) -- (d8) ;
\draw[dotted] (c3) -- (d7) (c3) -- (d8) ;
\color{blue}
\draw[thick] (a1) -- (c4) (b1) -- (d4) ; 
\draw[thick] (a4) -- (c7) (b4) -- (d7) ;
\draw (a2) -- (c7) ;
\draw[dashed] (a0) to[bend right=3] (c4) (a0) -- (c5) (a0) to[bend right=3] (c6) (a0) -- (c7) (a0) to[bend right=3] (c8) ;
\draw[dashed] (a1) to[bend right=3] (c5) (a1) -- (c6) (a1) to[bend right=3] (c7) (a1) -- (c8) ;
\draw[dashed] (a2) to[bend right=3] (c8) (a3) to[bend right=5] (c7) (a3) -- (c8) (a4) to[bend right=5] (c8) ;
\draw (b2) -- (d7) ;
\draw[dotted] (b0) to[bend right=3] (d4) (b0) -- (d5) (b0) to[bend right=3] (d6) (b0) -- (d7) (b0) to[bend right=3] (d8) ;
\draw[dotted] (b1) to[bend right=3] (d5) (b1) -- (d6) (b1) to[bend right=3] (d7) (b1) -- (d8) ;
\draw[dotted] (b2) to[bend right=3] (d8) (b3) to[bend right=5] (d7) (b3) -- (d8) (b4) to[bend right=5] (d8) ;
\color{green!60!black}
\draw[thick] (a2) to[bend right=5] (d5) ; 
\draw[thick] (a5) to[bend right=5] (d8) ; 
\draw[dashed] (a0) -- (d5) (a0) to[bend right=3] (d6) (a0) -- (d7) (a0) -- (d8) ;
\draw[dashed] (a1) -- (d5) (a1) to[bend right=3] (d7) (a1) -- (d6) (a1) -- (d8) ;
\draw[dashed] (a2) -- (d6) (a2) to[bend right=3] (d8) (a2) -- (d7) (a3) -- (d8) (a4) -- (d8) ;
\end{tikzpicture}
\qquad\raise20ex\hbox{\huge\dots~}
$$
$$
\begin{tikzpicture}[scale=0.38]
\scriptsize
\tikzstyle{every node}=[draw, shape=circle, minimum size=3pt,inner sep=0pt, fill=black]
\node[label=left:$c_1^0$] at (0,0) (a0) {}; \node at (0,1) (a1) {}; \node at (0,2) (a2) {};
\node at (0,3) (a3) {}; \node at (0,4) (a4) {}; \node at (0,5) (a5) {}; 
\node at (0,6) (a6) {}; \node at (0,7) (a7) {}; \node[label=left:$c_1^8$] at (0,8) (a8) {}; 
\node[label=left:$c_2^0$] at (3,0) (b0) {}; \node at (3,1) (b1) {}; \node at (3,2) (b2) {};
\node at (3,3) (b3) {}; \node at (3,4) (b4) {}; \node at (3,5) (b5) {}; 
\node at (3,6) (b6) {}; \node at (3,7) (b7) {}; \node[label=left:$c_2^8$] at (3,8) (b8) {}; 
\draw[thick] (a0) -- (a8) (b0) -- (b8) ; 
\draw[thick] (a0) -- (b3) ;
\draw[thick] (a3) -- (b6) ;
\draw (a1) -- (b6) ;
\draw (a0) -- (b4) (a0) -- (b5) (a0) -- (b6) (a0) -- (b7) (a0) -- (b8) ;
\draw (a1) -- (b7) (a1) -- (b8) ;
\draw (a2) -- (b6) (a2) -- (b7) (a2) -- (b8) ;
\draw (a3) -- (b7) (a3) -- (b8) ;
\end{tikzpicture}
\qquad
\begin{tikzpicture}[scale=0.38]
\scriptsize
\tikzstyle{every node}=[draw, shape=circle, minimum size=3pt,inner sep=0pt, fill=black]
\node[label=left:$c_1^0$] at (0,0) (a0) {}; \node at (0,1) (a1) {}; \node at (0,2) (a2) {};
\node at (0,3) (a3) {}; \node at (0,4) (a4) {}; \node at (0,5) (a5) {}; 
\node at (0,6) (a6) {}; \node at (0,7) (a7) {}; \node[label=left:$c_1^8$] at (0,8) (a8) {}; 
\node[label=left:$c_3^0$] at (6,0) (c0) {}; \node at (6,1) (c1) {}; \node at (6,2) (c2) {};
\node at (6,3) (c3) {}; \node at (6,4) (c4) {}; \node at (6,5) (c5) {}; 
\node at (6,6) (c6) {}; \node at (6,7) (c7) {}; \node[label=right:$c_3^8$] at (6,8) (c8) {}; 
\tikzstyle{every node}=[draw, shape=circle, minimum size=1.7pt,inner sep=0pt, fill=black]
\node[label=left:$c_2^0$] at (3,0) (b0) {}; \node at (3,1) (b1) {}; \node at (3,2) (b2) {};
\node at (3,3) (b3) {}; \node at (3,4) (b4) {}; \node at (3,5) (b5) {}; 
\node at (3,6) (b6) {}; \node at (3,7) (b7) {}; \node[label=left:$c_2^8$] at (3,8) (b8) {}; 
\draw[thick] (a0) -- (a8) (c0) -- (c8) ; 
\draw (b0) -- (b8) ;
\color{blue}
\draw[thick] (a1) -- (c4) ;
\draw[thick] (a4) -- (c7) ;
\draw (a2) -- (c7) ;
\draw (a0) to[bend right=3] (c4) (a0) -- (c5) (a0) to[bend right=3] (c6) (a0) -- (c7) (a0) to[bend right=3] (c8) ;
\draw (a1) to[bend right=3] (c5) (a1) -- (c6) (a1) to[bend right=3] (c7) (a1) -- (c8) ;
\draw (a2) to[bend right=3] (c8) (a3) to[bend right=5] (c7) (a3) -- (c8) (a4) to[bend right=5] (c8) ;
\end{tikzpicture}
\qquad
\begin{tikzpicture}[scale=0.38]
\scriptsize
\tikzstyle{every node}=[draw, shape=circle, minimum size=3pt,inner sep=0pt, fill=black]
\node[label=left:$c_1^0$] at (0,0) (a0) {}; \node at (0,1) (a1) {}; \node at (0,2) (a2) {};
\node at (0,3) (a3) {}; \node at (0,4) (a4) {}; \node at (0,5) (a5) {}; 
\node at (0,6) (a6) {}; \node at (0,7) (a7) {}; \node[label=left:$c_1^8$] at (0,8) (a8) {}; 
\node[label=right:$c_4^0$] at (9,0) (d0) {}; \node at (9,1) (d1) {}; \node at (9,2) (d2) {};
\node at (9,3) (d3) {}; \node at (9,4) (d4) {}; \node at (9,5) (d5) {}; 
\node at (9,6) (d6) {}; \node at (9,7) (d7) {}; \node[label=right:$c_4^8$] at (9,8) (d8) {};
\tikzstyle{every node}=[draw, shape=circle, minimum size=1.7pt,inner sep=0pt, fill=black]
\node[label=left:$c_3^0$] at (6,0) (c0) {}; \node at (6,1) (c1) {}; \node at (6,2) (c2) {};
\node at (6,3) (c3) {}; \node at (6,4) (c4) {}; \node at (6,5) (c5) {}; 
\node at (6,6) (c6) {}; \node at (6,7) (c7) {}; \node[label=right:$c_3^8$] at (6,8) (c8) {}; 
\node[label=left:$c_2^0$] at (3,0) (b0) {}; \node at (3,1) (b1) {}; \node at (3,2) (b2) {};
\node at (3,3) (b3) {}; \node at (3,4) (b4) {}; \node at (3,5) (b5) {}; 
\node at (3,6) (b6) {}; \node at (3,7) (b7) {}; \node[label=left:$c_2^8$] at (3,8) (b8) {}; 
\draw[thick] (a0) -- (a8) (d0) -- (d8) ; 
\draw (b0) -- (b8) (c0) -- (c8) ;
\color{green!60!black}
\draw[thick] (a2) to[bend right=5] (d5) ; 
\draw[thick] (a5) to[bend right=5] (d8) ; 
\draw (a0) -- (d5) (a0) to[bend right=3] (d6) (a0) -- (d7) (a0) -- (d8) ;
\draw (a1) -- (d5) (a1) to[bend right=3] (d7) (a1) -- (d6) (a1) -- (d8) ;
\draw (a2) -- (d6) (a2) to[bend right=3] (d8) (a2) -- (d7) (a3) -- (d8) (a4) -- (d8) ;
\end{tikzpicture}
$$
\caption{An illustration of the graph $G$ from Proposition~\ref{prop:propermix-to-tww-lower} for~$k=3$ and $h=4k-4=8$;
	the thick vertical chains form a clique each, and only $4$ out of all $2k+1=7$ chains are shown.
	The black slant edges depict those defined for $a=0$, the blue edges those for $a=1$ and the green edges those for $a=2$
	(the meaning of dashed/dotted edges is just to keep the picture tidy and not obscured by too many solid lines).
	The depicted pattern is cyclically shifted on consecutive $(k+1)$-tuples of the $2k+1$ chains.
	Below: a detail of the edge pattern between selected pairs of the chains.}
\label{fig:example-high-tww}
\end{figure}
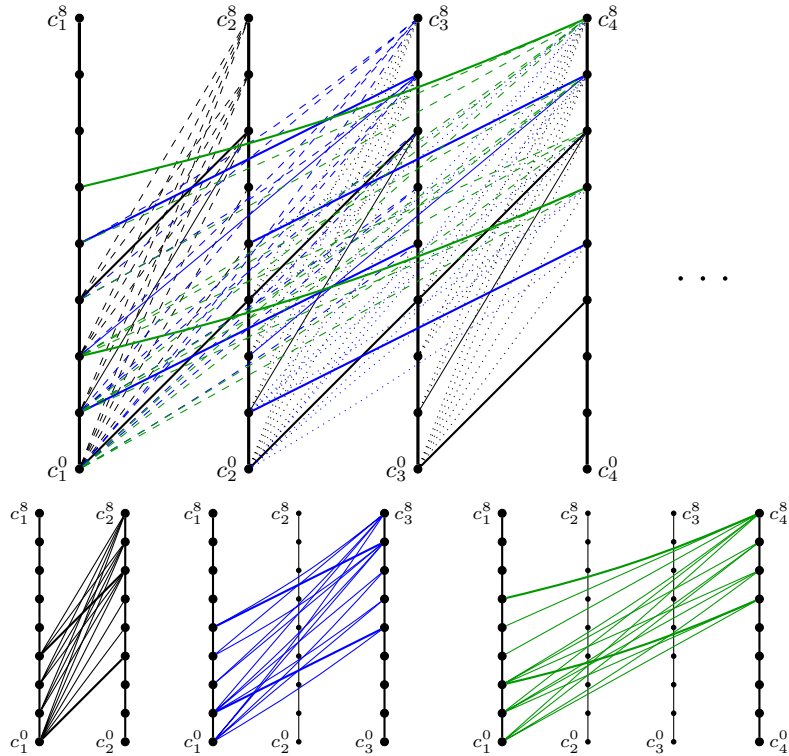
For $h\geq4k-4$, we construct the graph $G$ on $n=(2k+1)(h+1)$ vertices, where $V(G) = C_1\cup C_2\cup\ldots\cup C_{2k+1}$ and each $C_i=\{c_i^0,c_i^1,\ldots,c_i^h\}$ induces a clique.
In Figure~\ref{fig:example-high-tww}, we picture these cliques as the thick vertical chains.
When indexing these cliques, we take indices modulo $2k+1$, i.e., we declare $C_{2k+2}=C_1$, $C_{2k+3}=C_2$,\,\dots\
The edge set of $G$ is formed by the edges of these cliques, and by all following vertex pairs of~$V(G)$;
for $1\leq i\leq2k+1$, $0\leq a< k$ and $0\leq b <\lfloor(h-a)/k\rfloor$, we have $\{c_i^j,c_{i+a+1}^{j'}\}\in E(G)$,
if and only if $j\geq0$, $(b-1)k+a< j\leq bk+a$ and $(b+1)k+a\leq j'\leq h$.
For an illustration, see the slant edges in Figure~\ref{fig:example-high-tww}.

To prove that $G$ is inversion-free proper $(2k+1)$-mixed-thin, we use the partition $(C_1,\ldots,C_{2k+1})$, and define respective linear orders as follows;
on $C_i\cup C_{i'}$ where $i'=i+a+1$ (the indices are as above in the definition of $E(G)$\,), the order $\leq_{ii'}$ starts with the subsequence of vertices $c_{i'}^0,\ldots,c_{i'}^{k+a-1}$ as the least elements,
then we ``switch sides'' to continue with $c_i^{0},\ldots,c_i^{a}$, then with $c_{i'}^{k+a},\ldots,c_{i'}^{2k+a-1}$, then $c_i^{a+1},\ldots,c_i^{k+a}$, and so on \dots\ up to the highest elements $c_i^{h-k-2+a},\ldots,c_i^h$.
This clearly satisfies Definition~\ref{def:mixThin}.

To prove the lower bound on the twin-width of~$G$, we simply show that for every pair of vertices, one has at least $k$ neighbours which are not in the neighbourhood of the other
(and hence already the first contraction makes the red number high).
First consider two vertices $x,y$ coming from distinct chains; up to symmetry, we may assume that $x\in C_i$ and $y\in C_j$ where $j\in\{i+1,\ldots,i+k\}$ (modulo $2k+1$).
Recall that $x$ is adjacent to the rest of $C_i$ and $y$ to the rest of~$C_j$.
By the definition of $E(G)$ above, the $k$ vertices of $D=\{c_i^{h-k+1},\ldots,c_i^h\}\subseteq C_i$ are not in the neighbourhood of $C_j$, and so we are done unless $x\in D$.
In the latter case, we observe that $x$ has no neighbour in~$C_j$, which is also sufficient.

Now consider $x\not=y\in C_i$, such that $x=c_i^m$ and $y=c_i^{m'}$ where~$m'>m$. 
We may also assume $m\leq h-2k+1$, or we apply the symmetric argument (informally, in view of Figure~\ref{fig:example-high-tww}, with the graph ``turned upside down'').
For $a=m\!\mod k$, we observe $x$ has $k$ neighbours in $D'=\{c_{i+a+1}^{m+k},\ldots,c_{i+a+1}^{m+2k-1}\}$, but no vertex of $D'$ is a neighbour of~$y$.
We are again done.
\qed
\end{proof}

\section{Transduction Equivalence to Posets of Bounded Width}\label{sec:trans}

In relation to the deep fact~\cite{DBLP:conf/focs/Bonnet0TW20} that the class property of having bounded twin-width is preserved under FO transductions (cf.~Section~\ref{sec:prel}),
it is interesting to look at how our class of proper $k$-mixed-thin graphs relates to other studied classes of bounded twin-width.
In this regard we show that our class is nearly (note the inversion-free assumption!) transduction equivalent to the class of posets of bounded width.
We stress that the considered transductions here are always non-copying (i.e., not ``expanding'' the ground set of studied structures).

\begin{theorem}\label{theorem:from-posets}
The class of inversion-free proper $k$-mixed-thin graphs is a transduction of the class of posets of width at most $5\cdot \binom{k}{2} + 2k$.
For a given graph, together with the vertex partition and the orders as from Definition~\ref{def:mixThin}, the corresponding poset and its transduction parameters can be computed in polytime.
\end{theorem}

\begin{proof}
Let $G = (V,E)$ be an inversion-free proper $k$-mixed-thin graph. Let $\ca V = (V_1, \mathellipsis, V_k)$ be the partition of $V$ and $\leq_{ij}$ for $1 \leq i \leq j \leq k$ be the orders given by Definition~\ref{def:mixThin}. 
On a suitable ground set $X\supseteq V$ defined below, we are going to construct a poset $P = (X, \preceq)$ equipped with vertex labels ({\em marks}), such that the edges of $G$ will be interpreted by a binary FO formula within~$P$.
To simplify notation, we will also consider posets as special digraphs, and naturally use digraph terms for them.

For start, let $P_0 = (V, \preceq_0)$ be the poset formed by (independent) chains $V_1,\mathellipsis,V_k$, where each chain $V_i$ is ordered by $\leq_{ii}$. 
Let us denote by $V_{i,j}:=V_i\cup V_j$.

In order to define set $X$, we first introduce the notion of {\em connectors}.
Consider $1\leq i\leq j\leq k$, $X\supsetneq V$, a vertex $x\in X\setminus V$ and a pair $l_x\in V_i$ and $u_x\in V_j$. If $i=j$, we additionally demand~$l_x\lneq_{ii}u_x$.
If $\sqsubseteq_x$ is a binary relation (on $X$) defined by $l_x\sqsubseteq_x x\sqsubseteq_x u_x$, then we call $(x,\sqsubseteq_x)$ a {\em connector} with the {\em center $x$} and the {\em joins $l_xx$ and $xu_x$}.
(Note that it will be important to have $u_x$ from $V_j$ and not from $V_i$, wrt.~$i\lneq j$.)
We also order the connector centers $x\not=y$ with joins to $V_i$ and $V_j$ by $x\sqsubseteq_{ij}y$, if and only if $l_x\lneq_{ii}l_y$~and~$u_x\lneq_{jj}u_y$.
There may be more that one connector connecting the same pair of vertices.

Our construction relies on the following observation which, informally, tells us that connectors can (all together) encode some information about pairs of vertices of $V$ in an unambiguous way. 

\begin{claim2rep}\label{clm:getposet}
Recall $P_0 = (V, \preceq_0)$. Let $X\supsetneq V$ be such that each $x\in X\setminus V$ is the center of a connector, as defined above.
Let $\preceq_1$ be a binary relation on $X\supsetneq V$ defined as the reflexive and transitive closure of $(\preceq_0\cup\sqsubseteq)$ where
\mbox{$\sqsubseteq\,:=\big(\bigcup_{x\in X\setminus V}\sqsubseteq_x \!\big)$} $\cup\big( \bigcup_{1\leq i\leq j\leq k}\sqsubseteq_{ij} \!\big)$.
Then $P_1 = (X, \preceq_1)$ is a poset, \smallskip and each join of every connector $(x,\sqsubseteq_x)$ from $x\in X\setminus V$ is a cover pair in~$P_1$.
\end{claim2rep}
\begin{proof}[Subproof]
Let $D$ be the digraph on the vertex set $X$ and the arcs defined by the pairs in $(\preceq_0\cup\sqsubseteq)$.
A pair $(x,y)$ is in $\preceq_1$ if and only if there exists a directed path in $D$ from $x$ to~$y$.
It is routine to verify that $D$ is acyclic, in particular since there is no directed path starting in some $V_j$ and ending in $V_i$ where~$i<j$.
This implies that $\preceq_1$ is antisymmetric, and hence forming a poset.

For the second part, consider a connector $(x,\sqsubseteq_x)$ with the joins $l_xx$ and $xu_x$ to $V_i$ and $V_j$, and for a contradiction assume that $l_x\precneqq_1z\precneqq_1x$ for some~$z\in X$.
So, there exists a directed path $R$ in $D$ from $z$ to~$x$, and the only incoming arcs to $x$ are from $l_x\in V_i$ and from other connectors below $x$ in $\sqsubseteq_{ij}$.
If $R$ intersects $l_x$ or a vertex below $l_x$ in $\leq_{ii}$, we have a contradiction with the acyclicity of $D$. Otherwise all vertices of $R$ are connectors between $V_i$ and $V_j$ below $x$, which is again a contradiction with $l_x\precneqq_1 z$.
The case with possible $x\precneqq_1z\precneqq_1u_x$ is finished symmetrically.
\subqed
\end{proof}

We continue with the construction of the poset $P$ encoding $G$; this is done by adding suitable connectors to $P_0$, and marks $\bss$, $\bvi$, $\bbij$, or $\bcij$.
To explain, $\bss$ stands for successor (cf.~$\leq_{ij}$), $\bvi$ stands for the part $V_i$, $\bbij$ means a border-pair (to be defined later in $G[V_{i,j}]$), and $\bcij$ stands for complement (cf.~$E_{i,j}=\bar E$).
\begin{enumerate}
\item We apply the mark $\bvi$ to every vertex of each part~$V_i\in\ca V$.
\item For each $1\leq i< j\leq k$, and every pair $(v,w)\in V_i\times V_j$ such that $w$ is the immediate successor of $v$ in $\le_{ij}$,
	we add a connector with a new vertex $x$ marked $\bss$ and joins to $l_x=v$ and $u_x=w$.
	Note that one could think about symmetrically adding connectors for $w$ being the immediate predecessor, but these can be uniquely recovered from the former connectors.
\item For $1\leq i\leq j\leq k$ and $v,w\in V_{i,j}$, let $V_{i,j}[v,w]:=\{x\in V_{i,j}: v\leq_{ij}x\leq_{ij}w\}$ be a consecutive subchain, and
	call the set $V_{i,j}[v,w]$ {\em homogeneous} if, moreover, every pair of vertices between $V_i\cap V_{i,j}[v,w]$ and $V_j\cap V_{i,j}[v,w]$ is an edge in~$E_{i,j}$.
	(In particular, for $i=j$, homogeneous $V_{i,i}[v,w]$ means a clique in $G$ if $E=E_{i,i}$ or an independent set of $G$ otherwise.)
	If $V_{i,j}[v,w]$ is an inclusion-maximal homogeneous set in~$V_{i,j}$, then we call $(v,w)$ a {\em border pair} in $V_{i,j}$, and
	we add a connector with a new vertex $x$ marked $\bbij$ and joins to $v$ and $w$.
	Specifically, it is $l_x=v$ and $u_x=w$, unless $v\in V_j$ and $w\in V_i$ in which case $l_x=w$ and $u_x=v$.
\item For $1\leq i\leq j\leq k$, if $E_{i,j}=\bar E$, then we mark just any vertex by $\bcij$.
\end{enumerate}

Now we define the poset $P = (X, \preceq)$, where the set $X\supseteq V$ results from adding all marked connector centers defined above to $P_0 = (V, \preceq_0)$,
and $\preceq$ is the transitive closure of $(\preceq_0\cup\sqsubseteq)$ as defined in Claim~\ref{clm:getposet} for the added connectors.
See a simple example in Figure~\ref{fig:from-posets}.
\begin{figure}[t]
\centering
\includegraphics[page=1, width=0.55\hsize]{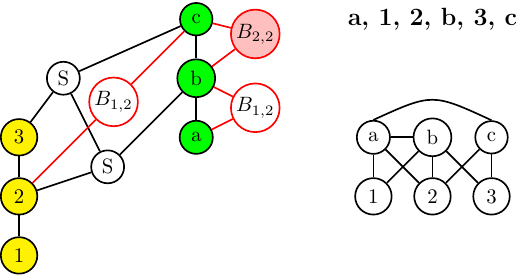}
\caption{An example of an inversion-free proper $2$-mixed-thin graph $G$ (bottom right), with parts $V_1 = \{1,2,3\}$ and $V_2 = \{a,b,c\}$. The ordering $\leq_{1,2}$ is top right ($\leq_{1,1}$ and $\leq_{2,2}$ are implied because the graph is inversion-free). $E_{1,2} = E_{1,1} = E$, and $E_{2,2} = \bar E$.\\
On the left, there is a Hasse diagram of a poset representing $G$. Vertices marked by $\boldsymbol{V_1}$ are coloured yellow, vertices marked by $\boldsymbol{V_2}$ are coloured green, and vertices marked by $\boldsymbol{C_{2,2}}$ are coloured pink (this mark encodes that $E_{2,2} = \bar E$, and we used it to mark an arbitrary vertex). Other marks (i.e., $\bss$, $\boldsymbol{B_{1,2}}$ or $\boldsymbol{B_{2,2}}$) are written inside the vertices.}
\label{fig:from-posets}
\vspace*{-10pt}
\end{figure}

First, we claim that $P$ with the applied marks uniquely determines our starting graph~$G$.
Notice that, for each connector center $x\in X\setminus V$, the (unique) cover pairs of $x$ to and from respective $V_i$ and $V_j$,
by Claim~\ref{clm:getposet}, determine the joins~of~$x$.

The vertex set of $G$ is determined by the marks $\bvi$,~$i=1,\ldots,k$.
For $1\leq i\leq j\leq k$, the linear order $\leq_{ij}$ is directly determined by $\preceq$ if $i=j$, and otherwise the following holds.
For $v\in V_i$ and $w\in V_j$, we have $v\leq_{ij}w$ if and only if there exists a connector $x$ marked $\bss$ with joins to $l_x\in V_i$ and $u_x\in V_j$ such that $v\preceq l_x$ and $u_x\preceq w$.
For $v\in V_j$ and $w\in V_i$, we have $v\leq_{ij}w$ if and only if $w\not\leq_{ij}v$.

To determine the edge set of~$G$, we observe that Definition~\ref{def:mixThin} shows that every edge $f$ (resp. non-edge) of $G[V_{i,j}]$  is contained in some homogeneous consecutive subchain of $\leq_{ij}$.
Hence $f$ is contained in some maximal such subchain, and so determined by some border pair in $V_{i,j}$ which we recover from its connector marked $\bbij$ using the already determined order~$\leq_{ij}$. 
We then determine whether $f$ means an edge or a non-edge in $G$ using the mark $\bcij$.

Finally, we verify that the above-stated definition of the graph $G$ within $P$ can be expressed in FO logic.
We leave the technical details for the next claim:

\begin{claim2rep}
The transduction described in the proof of Theorem~\ref{theorem:from-posets} can be defined by FO formulae on the marked poset~$P$.
\end{claim2rep}
\begin{proof}[Subproof]
\newcommand{\vertex}{\varphi_0}
We start with the vertex formula $\vertex(u)$ (which is satisfied if and only if $u \in V$).
Then, we continue with a sequence of auxiliary formulae, leading us to the edge formula $\varphi_E(u,v)$ (which is satisfied if and only if $uv \in E$). Many of the auxiliary formulae are parameterized by $i$ and $j$ for $1 \leq i,j \leq k$.

\[\vertex(u) ~\equiv \bigvee_{1 \leq i \leq k} \bvi(u)\]

The next formula simply says that the vertices $u$ and $v$ both belong to the part $V_i$, and that $u$ is smaller than $v$ in the internal ordering of $V_i$ (i.e., $\leq_{ii}$).

\[u \leq_{ii} v ~\equiv~ \bvi(u) \land \bvi(v) \land u \preceq v \]

The next formula says that $w$ is a center of a connector (see the definition above) connecting $u \in V_i$ and $v \in V_j$. The second line states that $(u,w)$ and $(w,v)$ are cover pairs.
\vspace*{-2pt}
\begin{equation*}
\begin{split}
Connect_{ij}(u,v,w) ~\equiv\> &\;\bvi(u) \land \bvj(v) \land \neg \vertex(w)\; \land
                             u \preceq w \land w \preceq v\; \land \\
                             &\neg \exists x\; (u \precneqq x \land x \precneqq w)\; \land
                             \neg \exists x\; (w \precneqq x \land x \precneqq v)\\
\end{split}
\end{equation*}

These three formulae decode the ordering $\leq_{ij}$, using the connectors marked by $\bss$. `$u \leq_{ij}' v$' is an auxiliary formula, which is equivalent to `$u\leq_{ij}v$' only if $u \in V_i$ and $v\in V_j$.
\begin{align*}
Succ_{ij}(u,v) ~\equiv~ \; &\exists w\; (Connect_{ij}(u,v,w) \land \bss(w))\\
u \leq_{ij}' v ~\equiv~ \; &\exists u^+, v^- \;( u \leq_{ii} u^+ \land v^- \leq_{jj} v \land Succ(u^+, v^-))  
\end{align*}
\vspace{-10pt}
\begin{equation*}
\begin{split}
u \leq_{ij} v ~\equiv\, \;  u \leq_{ii} v \lor u \leq_{jj} v &\lor (\bvi(u) \land \bvj(v) \land u \leq_{ij}' v)\\ &\lor (\bvj(u) \land \bvi(v) \land \neg (v \leq_{ij}' u))
\end{split}
\end{equation*}

\medskip
This formula says that $(u,v) \in V_{i,j}^{~2}$ is a border-pair, see the definition above.
\vspace*{-2pt}
\begin{equation*}
\begin{split}
BorderPair_{ij}(u,v) ~\equiv~ \; &u \leq_{ij} v \land \exists w\;\big[\bbij(w)\; \land\\
\big[&(\bvi(u) \land \bvi(v) \land Connect_{ii}(u,v,w))\;\lor \\
 &(\bvj(u) \land \bvj(v) \land Connect_{jj}(u,v,w))\;\lor\\
 &(\bvi(u) \land \bvj(v) \land Connect_{ij}(u,v,w))\;\lor\\
 &(\bvj(u) \land \bvi(v) \land Connect_{ij}(v,u,w))\big]\big]
\end{split}
\end{equation*}

\medskip
The formula $\psi$ says that there is an edge in $G$ between vertices $u$ and $v$ such that $u \in V_i$ and $v \in V_j$ for some $1\leq i, j \leq k$, and $u \leq_{ij} v$. Thus, it is almost the edge formula, except that it may be satisfied for $u = v$, and it may not be satisfied after swapping $u$ and $v$. These issues are fixed by the formula $\varphi_E$ itself.

Intuitively, $\psi(u,v)$ means that $u$ and $v$ are in the homogeneous set defined by the border-pair $(u^-,v^+)$, which by the definition of a border-pair means that $uv$ is an edge in $G$ (or a non-edge, depending on the respective $E_{ij}$, which is why the $\exists x\; (\bcij(x))$ part is there). The symbol $\oplus$ stands for the exclusive disjunction.
\begin{equation*}
\begin{split}&\psi(u,v) ~\equiv \bigvee_{1\leq i, j \leq k} \bvi(u) \land \bvj(v)\; \land u \leq_{ij} v\; \land\\
                           &\qquad\big(\exists x\; (\bcij(x)) \oplus \exists u^-, v^+ (u^- \leq_{ij} u \land  v \leq_{ij} v^+ \land BorderPair_{ij}(u^-,v^+))\big)
\end{split}
\end{equation*}

\[\varphi_E(u,v) ~\equiv~ u \neq v \land (\psi(u,v) \lor \psi(v,u))
\eqno{\subqed}  \]
\end{proof}

\smallskip
\ifx\proof\inlineproof
Continuing in the proof of Theorem~\ref{theorem:from-posets}, we now compute the width of $P$. 
\else
Second, we compute the width of $P$. 
\fi
In fact, we show that $P$ can be covered by a small number of chains. There are the $k$ chains of~$V_1,\ldots,V_k$.
Then, for each pair $1\leq i < j\leq k$, we have one chain of the connector centers marked $\bss$ from $V_i$ to $V_j$,
and four chains of the connector centers marked $\bbij$, sorted by how their border pairs fall into the sets $V_i$ or~$V_j$ (they are indeed chains because border pairs demarcate maximal homogeneous sets), thus $5\cdot \binom{k}{2}$ chains. Finally, there is a chain of the connector centers marked $\mx {B_{ii}}$ for each $1 \le i \le k$.
To summarize, there are $k+ 5\cdot \binom{k}{2} + k$ chains covering whole~$P$.

\smallskip
Efficiency of the construction of marked poset $P$ from given (already partitioned and with the orders) graph $G$ is self-evident.
The whole proof of Theorem~\ref{theorem:from-posets} is now finished.
\qed
\end{proof}

Now we prove the converse statement to Theorem~\ref{theorem:from-posets}, i.e., we show how to represent posets in inversion-free proper mixed-thin graphs.

\begin{theorem2rep}\label{theorem:to-posets}
The class of posets of width at most $k$ is a transduction of the class of inversion-free proper $(2k+1)$-mixed-thin graphs.
For a given poset, a corresponding inversion-free proper $(2k+1)$-mixed-thin graph can be computed in polytime.
\end{theorem2rep}
\begin{proof}
Let $P = (X,\preceq)$ be a poset of width $k$. Let us fix a partition $(C_1,\mathellipsis,C_k)$ of $P$ into $k$ chains. We construct a graph $G = (V,E)$ augmented by marks $\baa$, $\bbb$, $\boo$ and $\bci$ for $1 \leq i \leq k$, and we show that $P$ can be interpreted in $G$, and that $G$ is an inversion-free proper ($2k+1$)-mixed-thin graph. Let us denote by $h$ the length of the longest chain in $P$, i.e., $h = \max(|C_1|,\mathellipsis,|C_k|)$.

For $1 \leq i \leq k$ and for each $v \in C_i$, there are two vertices $a_v$ and $b_v$ in $V$; both of them are marked by $\bci$, $a_v$ is marked by $\baa$, and $b_v$ is marked by $\bbb$. Furthermore, there are $h$ additional vertices in $V$, which we denote $O = \{o_1,\mathellipsis,o_h\}$. These additional vertices are marked by $\boo$, and they are used to encode the (internal) orderings of the chains. Now let us define the edge set $E$ (see Figure~\ref{fig:to-posets}):

\begin{figure}[t]
\centering
\includegraphics[page=1,scale=1.05]{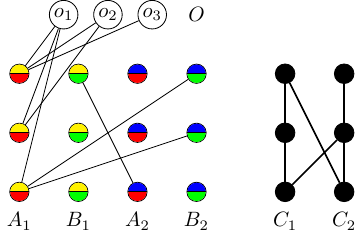}
\caption{On the right, there is a Hasse diagram of a poset $P$, partitioned into two chains, $C_1$ and $C_2$. On the left, there is an inversion-free proper 5-mixed-thin graph representing $P$. Vertices coloured red are marked by $\baa$, vertices coloured green are marked by $\bbb$, vertices coloured yellow are marked by $\boldsymbol{C_1}$, and vertices coloured blue are marked by $\boldsymbol{C_2}$. The edges going from $O$ to $B_1$, $A_2$ and $B_2$ are omitted (they are similar to the edges between $O$ and $A_1$).}
\label{fig:to-posets}
\end{figure}

\begin{itemize}
\item For $1 \leq i \leq k$, $v \in C_i$ and $1 \leq m \leq h$, $\{o_m, a_v\} \in E$ and $\{o_m, b_v\} \in E$ if and only if $m \leq |\{u \in C_i : u \preceq v\}|$.
\item For $1 \leq i \neq j \leq k$, $u \in C_i$ and $v \in C_j$, $\{a_u,b_v\} \in E$ if and only if $u \preceq v$.
\item There are no other edges in $E$.
\end{itemize}

Now let us prove that $G$ is an inversion-free proper ($2k+1$)-mixed-thin graph. Let the partition required by Definition~\ref{def:mixThin} be $\mpp = (V_1 = O, V_2= A_1,\mathellipsis,V_{k+1}=A_k,V_{k+2}=B_1,\mathellipsis,V_{2k+1}=B_k)$ where $A_i = \{a_v: v \in C_i\}$ and $B_i = \{b_v: v \in C_i\}$ for $1 \leq i \leq k$. Now we define the ordering $\leq_{rs}$ for $1\leq r,s \leq 2k+1$.

\begin{itemize}
\item If $r=s$, then for $V_r = O$, $o_i \leq_{rr} o_j$ if and only if $i \leq j$ (for all $1\leq i,j \leq h$), and for $V_r = A_i$ (resp. $V_r = B_i$) for some $1\leq i \leq k$, $a_u \leq_{rr} a_v$ (resp. $b_u \leq_{rr} b_v$) for $u,v \in X$ if and only if $u \preceq v$.
\item If $V_r = O$, then $\leq_{rs}$ is obtained by alternately taking vertices of $O$ and $V_s$, starting with $o_1$. For example, if $V_s = A_i$ for some $1\leq i\leq k$ and $C_i$ ordered by $\preceq$ equals $(v_1,\mathellipsis,v_n)$, then $O \cup A_i$ ordered by $\leq_{rs}$ is $(o_1, a_{v_1}, o_2,\mathellipsis, a_{v_n},\\o_{n+1}, o_{n+2},\mathellipsis,o_h)$.
\item If $V_r= A_i$ and $V_s=B_j$ for $1 \leq i \neq j \leq k$, then for $u \in C_i$ and $v \in C_j$, $a_u \leq_{rs} b_v$ if and only if $u \preceq v$ (note that if $u$ and $v$ are incomparable in $\preceq$, then $b_v \leq_{rs} a_u$ because $\leq_{rs}$ is total). 
\item Otherwise, $\leq_{rs}$ is irrelevant because there are no edges between $V_r$ and $V_s$.
\end{itemize}

Now we prove that these orderings satisfy Definition~\ref{def:mixThin}. Let $1\leq r, s \leq 2k+1$. We set $E_{r,s} = E$. Observe that for $u \in V_r$ and $w \in V_s$, $u \leq_{rs} w$ fully determines if $uw \in E$ or not, i.e.:

\begin{itemize}
\item for $V_r = O \neq V_s$, or $V_r=A_i$ and $V_s=B_j$ for $1 \leq i \neq j \leq k$, $u \leq_{rs} w$ if and only if $uw \in E$.
\item for $V_s = O\neq V_r$, or $V_r=B_i$ and $V_s=A_j$ for $1 \leq i \neq j \leq k$, $u \leq_{rs} w$ if and only if $uw \notin E$.
\item for other choices of $r$ and $s$, we know that $uw \notin E$.
\end{itemize}

This observation immediately implies that $G$ is a proper $(2k+1)$-mixed-thin graph because for $u \leq_{rs} v \leq_{rs} w$ and $u,v \in V_r$, $w \in V_s$ (resp. $u \in V_r$; $v,w \in V_s$), it may not occur that $uw \in E$ and $vw \notin E$ (resp. $uv \notin E$).

\medskip
Finally, we show that $P$ can be interpreted in $G$. The first auxiliary formula, $SameChain(u,v)$, is satisfied if $u$ and $v$ represent elements of $P$ of the same chain. The second auxiliary formula, $u \leq_{internal} v$, expresses the internal orderings of the chains (it may be satisfied even if $u$ and $v$ represent elements of different chains but it does not matter since we use $\leq_{internal}$ only if $SameChain(u,v)$ is satisfied, see $Twins(u,v)$ and $u \preceq v$). The third auxiliary formula, $Twins(u,v)$, is satisfied if $u$ and $v$ represent the same element of $P$, i.e., $u,v \in \{a_x,b_x\}$ for some $x \in X$.

$\varphi_0(u)$ says which vertices of $G$ are elements of $P$ (note that if we used $\bbb$ instead of $\baa$, the interpretation would give the same poset). Finally, $u \preceq v$ expresses the partial order of $P$.
\begin{equation*}
\begin{split}
SameChain(u,v) &~\equiv \bigvee_{1\leq i \leq k} \bci(u) \land \bci(v)\\
u \leq_{internal} v &~\equiv~ \forall w\; ((\boo(w) \land wu \in E) \rightarrow wv \in E)\\
Twins(u,v) &~\equiv~ SameChain(u,v) \land u \leq_{internal} v \land v \leq_{internal} u\\
\varphi_0(u) &~\equiv~ \baa(u)
\end{split}
\end{equation*}
\begin{equation*}
\begin{split}
u \preceq v ~\equiv~\; &(SameChain(u,v) \rightarrow u \leq_{internal} v)\;\land\\
&(\neg SameChain(u,v) \rightarrow\\
&\exists u',v' (Twins(u,u') \land Twins(v,v') \land \baa(u') \land \bbb(v') \land u'v' \in E))\\
\end{split}
\end{equation*}

\qed

\end{proof}

\section{The Red Potential Method}\label{sec:red-potential}

In the proof of Theorem~\ref{thm:propermix-to-tww}, as well as in \cite{DBLP:conf/iwpec/BalabanH21} before,
we have applied a useful proof technique estimating the average increase in red degrees over a selected subset of candidate contractions.
The purpose of this section is to introduce this technique in a general formulation, in hope that it will find its applications
in proving efficiently bounded twin-width of other classes.
We also outline some of the limits of applicability of this technique.

To approach the technique formally, we are going to define a red-potential property,
which will subsequently be used to efficiently obtain a desired contraction sequence (which is not necessarily optimal, but has a guaranteed red value).
This application, described by Proposition~\ref{prop:red-potential-method-algorithm}, is what we call the {\em Red potential method}.

\begin{definition}[Red potential]\label{def:red-potential-method}
    Let $\mx M$ be a symmetric {$(n \times n)$-matrix} with entries from a finite set containing the red entry $r$.
    We denote by $V(\mx M)$ the set of rows (or equivalently columns) of $\mx M$, and let $\preceq$ be an arbitrary linear order on $V(\mx M)$.
    For a subset $R\subseteq V(\mx M)$, we denote by $S_R^\preceq$ the successor (cover) relation of $\preceq$ restricted to~$R$,
    i.e., $S_R^\preceq$ ``makes'' a directed increasing path on~$R$.

    Denote by $P_{\mx M}(u, v)$ the number of red entries of the row created by contracting rows $u$ and $v$ in $\mx M$.
    The {\em red potential of the set $R$ in $\mx M$ with respect to~$\preceq$} is defined as
    $P^\preceq_{\mx M,R} = \sum_{(u, v) \in S_R^\preceq} P_{\mx M}(u, v)$.
    We shortly say the {\em red potential of $R$} when $\mx M$ and $\preceq$ are clear.
    
    If $G$ is an $n$-vertex graph with the vertex set $V(G)$ ordered by $\preceq$,
    then the {\em red potential of $R\subseteq V(G)$ in $G$ with respect to $\preceq$} is simply the red potential of
    $R$ in the adjacency matrix $\mx M:=\mx A(G)$ with respect to $\preceq$.
    Here we do not distinguish between vertices of $G$ and the corresponding rows (resp. columns) of $\mx M$, that is, $V(\mx M)=V(G)$.

\smallskip
    We say that such an $(n \times n)$-matrix $\mx M$ with a linear order $\preceq$ on $V(\mx M)$
    has the {\em$(k,\ell)$-red-potential property} if $n \le \ell$, or if all of the following hold:
    \begin{enumerate}
        \item\label{item:rpm-max-red-degree}
        the number of red entries (the red degree) in any row of $\mx M$ is at most $\ell$,
        \item\label{it:rpaverage}
	there exists a subset $R\subseteq V(\mx M)$ such that the red potential of $R$ in $M$ with respect to $\preceq$
	is at most $(k+1)(|R|-1)-1$, and
        \item\label{it:rprecur}
	for every pair of rows $(u, v) \in S_R^\preceq$ such that $P_M(u, v) \le k$;
	if $\mx M'$ denotes the matrix obtained from $\mx M$ by contracting the rows $u$ with $v$ and then contracting the corresponding columns,
	then $\mx M'$ with the order inherited from $\preceq$ on $V(\mx M')$ has again the $(k,\ell)$-red-potential property.
    \end{enumerate}
\end{definition}

Definition~\ref{def:red-potential-method} deserves several important comments.
First, note that, in the $(k,\ell)$-red-potential property, condition \ref{it:rprecur}. immediately implies that $k\leq\ell$ is necessary.
However, it may not be apriori clear why we need both $k$ and $\ell$ in the definition.
This is best illustrated by the proof of Theorem~\ref{thm:propermix-to-tww}; while condition \ref{it:rpaverage}. guarantees that there
is a row pair $u,w$ to contract, such that the resulting red degree of the contracted row $w$ is at most $k$, 
this does not mean that in subsequent row contractions the corresponding column contractions do not increase the red degree of~$w$,
and the generally larger bound on the red degree $\leq\ell$ captures this possibility.

Second, observe that whenever $\mx M$ has the symmetric twin-width at most~$t$, then, for a suitable order $\preceq$
inherited from the assumed $t$-contraction sequence of $\mx M$, the $(t,t)$-red-potential property is satisfied by $\mx M$
(simply; we always choose as $R$ the pair of rows which is to be contracted in the assumed sequence).
However, this is of not much help since we are primarily interested in efficient ways of finding a bounded contraction sequence.
One should thus view Definition~\ref{def:red-potential-method} in a way that both the given order $\preceq$ and the choice(s)
of the set $R$ in condition~\ref{it:rpaverage}. somehow ``naturally'' follow from a given presentation of the graph~$G$
(which was the case of both \cite{DBLP:conf/iwpec/BalabanH21} and Theorem~\ref{thm:propermix-to-tww}).

Third, we comment a bit on the ``suitable order $\preceq$'' on $V(\mx M)$.
While \cite{DBLP:journals/jacm/BonnetKTW22} prove that a matrix $\mx M$ has bounded (symmetric) twin-width, if and only if
there exists a linear order on $V(\mx M)$ such that $\mx M$ does not contain a certain rather simple obstruction
(so-called mixed minor) with respect to~$\preceq$, the efficient construction accompanying this result is very complicated,
and generally raises the twin-width to a double-exponential function of the obstruction size.
Therefore, even only in special cases, it is valuable to have a straightforward method to construct contraction sequences
of reasonably small red degree, such as the one coming from Definition~\ref{def:red-potential-method} and described in the next statement:

\begin{proposition}\label{prop:red-potential-method-algorithm}
    Let $\mx M$ be a symmetric matrix, and $\preceq$ be a linear order on $V(\mx M)$.
    If $\mx M$ with $\preceq$ has the $(k,\ell)$-red-potential property for some $k\leq\ell$,
    then there is a symmetric $\ell$-contraction sequence of $\mx M$.
    Moreover, if the choice of the set $R$ in condition~\ref{it:rpaverage}. of Definition~\ref{def:red-potential-method}
    can be done in polynomial time, then an $\ell$-contraction sequence of $\mx M$ can also be found in polynomial time.
\end{proposition}
\begin{proof}
    This is straightforward by induction on~$n$:

    If $n \le \ell$, then any contraction sequence is fine.
    Otherwise, we choose the set $R\subseteq V(\mx M)$ as in condition~\ref{it:rpaverage}., and we know that
    $P^\preceq_{\mx M,R}\leq (k+1)(|R|-1)-1$.
    By the pigeon-hole principle, we thus have a successive pair $(u, v) \in S_R^\preceq$ such that $P_M(u, v) \le k$,
    and we choose such pair $u,v$ minimizing $P_M(u, v)$ over all $|R|-1$ possibilities in~$R$.
    After contracting $u$ and $v$, we obtain an $((n-1)\times(n-1))$-matrix $M'$ which has the $(k,\ell)$-red-potential property.
    We finish the desired sequence from $\mx M'$ by induction.
\qed\end{proof}

It is natural to ask whether all requirements of Definition~\ref{def:red-potential-method} are necessary.

We first present an indirect evidence that the recursive requirement of the red-potential property (condition~\ref{it:rprecur})
is necessary when one aims to use red potential to find full contraction sequence:

Let $\mx M$ be a symmetric $(n \times n)$-matrix ordered by $\preceq$, and choose $R \subseteq V(M)$. 
Assume that the red potential of $R$ in $\mx M$ is linearly bounded in $|R|$ (which implies
that there is an available consecutive contraction of constant red degree).
Let us attempt to create a contraction sequence for $\mx M$ by iteratively contracting pairs of consecutive rows
$u$ and $v$ (and the corresponding columns) such that $P_{\mx M}(u, v) \le c$ for some constant $c$.

Assume that we succeed, that is, we have contracted $R$ to a bounded number of rows. 
We still need to contract the rows $V(\mx M) \setminus R$. If we do not have any knowledge about
them, then the best we can do is an arbitrary symmetric contraction
sequence, therefore we need that $|V(\mx M) \setminus R|$ is bounded by constant. 
However, we might as well insist that $|V(\mx M) \setminus R| = 0$, since adding constant number of
rows to $R$ and obtaining $R'$ preserves that red potential of $R'$ in $\mx M$ is linear in $|R'|$.

We now show that such optimistic approach must fail:

\begin{proposition}
    There exists a class of graphs $\mathcal{B}$ such that 
    the twin-width of $\mathcal{B}$ is unbounded, and
    for every $3n$-vertex graph $G \in \mathcal{B}$ there exists
    an ordering $\preceq$ such that the red potential of $V(G)$
    in $G$ with respect to $\preceq$ is at most $10n$.
\end{proposition}
\begin{proof}
    Let $n \ge 1$.
    Let $\sigma_0$ and $\sigma_2$ be permutations on $n$ elements.
    Consider graph $G_{\sigma_0, \sigma_2}$ such that
    $V(G_{\sigma_0, \sigma_2}) = \{0,1,2\}\times\{1,2,\ldots,n\}$ and 
    $E(G_{\sigma_0, \sigma_2}) = \{(1, j)(k, i) : k \in\{0,2\}, 
                                                  \sigma_k(i) < j\}$.
    
    Denote by $\sigma_1$ the identity permutation, that is, 
    $\sigma_1(i) = i$ for all $i$.
    Consider ordering $\preceq$ such that
    $(i, j) \preceq (k, \ell)$ iff $i < k$ or 
    $i = k \land \sigma_i(j) \le \sigma_k(\ell)$.

    Let $\mx{A}_\preceq(G)$ be the adjacency matrix of $G$ ordered by $\preceq$.
    Let $k \in \{0, 2\}$ and observe that, for any two vertices $(k, i)$
    and $(k, j)$ such that $\sigma_k(i) = \sigma_k(j) + 1$, the 
    neighborhoods of $(k, i)$ and $(k, j)$ differ by at most one vertex.
    Furthermore, $(k, i)(k, j) \not\in E(G)$, hence
    $P_{\mx{A}_\preceq(G)}((k, i), (k,j)) \le 1$.
 
    Similarly, for any two vertices $(1, i)$ and $(1, i + 1)$, the 
    neighborhoods of $(1, i)$ and $(1, i+1)$ differ
    by at most two vertices.
    Furthermore, $(1,i)(1,i+1) \not\in E(G)$, hence
    $P_{\mx{A}_\preceq(G)}((1, i), (1,i+1))\le2$.

    There remain only two pairs of vertices that are in the successor relation
    in $\preceq$, and for each such pair $(u, v)$ we get that
    $P_{\mx{A}_\preceq(G)}(u, v) \le |V(G_{\sigma_0, \sigma_2})|=3n$.

    Together, we obtain that the red potential of $G$ with
    respect to $\preceq$ is at most $n+n+2n+2\cdot3n=10n$.

\smallskip
    Consider the class $\mathcal{B} = \{ G_{\sigma_0, \sigma_2} :
    \sigma_0\text{ and }\sigma_2$ are permutations on the same number of elements$\}$.
    Bonnet et al. \cite{DBLP:conf/soda/BonnetGKTW21} show that
    the class $\mathcal{B}$ has unbounded twin-width.
    Hence the class $\mathcal{B}$ has the desired properties.
\qed\end{proof}

Second, we address the question of whether the condition of recursively having the red degree
of $\mx M$ at most $\ell$ is truly necessary in the current form.
While we have already argued on the example of the proof Theorem~\ref{thm:propermix-to-tww}, that 
a bound generally larger than the bound of $k$ coming from the red-potential property is needed for $\mx M$,
it could still be possible that, whenever the red potential stays bounded as in Definition~\ref{def:red-potential-method} 
along the whole recursive procedure, the red degree of $\mx M$ stays bounded from above by a function of~$k$.

Again, this is not the case.
In nutshell, even if contractions of selected row pairs always result in rows with bounded number of red entries, 
the corresponding column contractions can increase the number of red entries in other rows beyond any control.
We show this under an additional requirement that $R=V(\mx M)$:

\begin{proposition}
    For every $\ell \ge k \ge 3$ there exists $n$ and an ordered $(n \times n)$-matrix
    $\mx M$ such that:
    \begin{itemize}
        \item[(a)] for every symmetric contraction sequence $\mx M= \mx M_n,\ldots,\mx M_{1}$
        and every $n \ge i \ge \ell$ we have that the red potential of $V(\mx M_i)$ in 
        the $(i \times i)$-matrix $\mx M_i$ is at most $3\cdot|V(\mx M_i)|$ with respect to any linear order; and
        \item[(b)] there is a symmetric contraction sequence $\mx M=\mx M_n,\ldots,\mx M_{1}$
        such that:
        \begin{itemize}
            \item[(i)] for all $i$, the matrix $\mx M_i$ has been created from $\mx M_{i+1}$
            by contracting consecutive rows $u$ and $v$ (and the corresponding 
            columns) such that $P_{\mx M_{i+1}}(u, v) \le k$,
            \item[(ii)] there exists $i$ such that the matrix $\mx M_i$ contains
            a row with more than $\ell$ red entries, and
            \item[(iii)] one could obtain this contraction sequence by applying
            the algorithm from Proposition~\ref{prop:red-potential-method-algorithm}
            when we choose $R$ to be $V(\mx M')$ for each matrix $\mx M'$ in the sequence.
            Note that, preconditions of the algorithm might not be satisfied.
        \end{itemize}
    \end{itemize}
\end{proposition}
\begin{proof}
    Let $n = 2\ell + 3$. Consider a graph $G$ such that
    $V(G) = \{v_1,v_2,\ldots,v_n\}$ and
    $E(G) = \{v_1v_{2k} : 1 \le k \le \frac{n-1}{2} \}$.
    Furthermore consider ordering $v_i \preceq v_j \iff i \le j$.
    Let $\mx M = \mx{A}_\preceq(G)$ be the adjacency matrix of $G$ ordered by $\preceq$.

    First, we prove that $\mx M$ satisfies (a). We consider matrix $\mx M'$ obtained 
    from $\mx M$ by any symmetric contraction sequence.
    Let $\preceq'$ be any linear order on $V(\mx M')$. At most two pairs of rows in
    $S^{\preceq}_{V(\mx M')}$ contain $v_1$ (or the row created by
    contracting $v_1$ with other rows). The remaining rows have only one non-zero
    entry. Therefore the red potential of $V(\mx M')$ in $\mx M'$ is 
    at most $2\cdot|V(\mx M')| + |V(\mx M')|\cdot1$.

    Second, we prove that $\mx M$ ordered by $\preceq$ satisfies (b).
    Notice that $P_{\mx M}(v_1, v_2) \ge \ell$, since vertex $v_{2i}$ is a neighbor
    of $v_1$ but not of $v_2$ for every $2 \le i \le \ell+1$.
    
    Consider a pair of rows $v_{i}, v_{i+1}$ in the successor relation in $\preceq$ 
    where $2 \le i < n$. The only nonzero entry in $v_{i}$ or $v_{i+1}$
    is in the column $v_1$, therefore $P_{\mx M}(v_i, v_{i+1})\le1$. Furthermore,
    $P_{\mx M}(v_i, v_{i+1})\ge1$ since contraction of rows $v_i$ and $v_{i+1}$ results
    in a red entry in the column $v_1$.
    Furthermore, this property is preserved by any number of contraction of
    rows $v_j$ and $v_{j+1}$ (and the corresponding columns) such that $2 \le i < n$.

    Therefore, an application of our algorithm can result in a contraction sequence
    beginning by contracting rows $v_{2i}$ and $v_{2i+1}$ (and the corresponding
    columns) for every $1 \le i \le \frac{n-1}{2}$.
    After these contractions are performed, $v_1$ has red entry in every
    column except for column $v_1$. Since there is $\ell+1$ other columns,
    such sequence satisfies requirements of (b)(ii).

    We know from (a) that any symmetric contraction sequence contains only matrices
    $\mx M'$ with red potential $3\cdot |V(\mx M')|$. Hence there is 
    a pair of consecutive rows $u, v$ such that $P_{\mx M'}(u, v) \le 3 \le k$, and
    the algorithm can extend the partial sequence to a full contraction sequence by
    choosing rows satisfying (b)(i).
\qed\end{proof}

Finally, we show that even an optimal twin-ordering of $\mx M$ might not be a suitable ordering for applying 
the red potential method on $\mx M$, again when we consider the additional requirement $R = V(\mx M)$.

\begin{proposition}
    For every $n \ge 3$, there exists a cograph $G_n$ on $n$ vertices
    and $1$-twin-ordering $\preceq$ of its adjacency matrix such that
    the red potential
    of $V(G)$ in $G_n$ with respect to $\preceq$ is at least
    $\frac{(n - 1)(n-2)}{2}$.
\end{proposition}
\begin{proof}
    Let $G_1:=K_1$, let $G_{2i}:=\overline{\overline{G_{2i-1}} \cupdot K_1}$,
    and let $G_{2i+1}:=G_{2i} \cupdot K_1$ for all $i \ge 1$, where $\bar H$ denotes the graph complement. Let us fix arbitrary $n \ge 3$, and let us number the vertices of $G_n$ in the order they were added (i.e., $v_i$ is present in $G_i$ but not in $G_{i-1}$).
    Consider the adjacency matrix of $G_n$ ordered by $\preceq$.

    Observe that the contraction sequence obtained by
    contracting $v_1$ with $v_2, v_3,\\ \ldots, v_n$ (in this order)
    creates no red entries except for self-loops.
    Hence the ordering $v_i \preceq v_j \iff i \le j$ is a $1$-twin-ordering.
    Note that there is no $0$-twin-ordering of $G_n$ since any
    contraction sequence must eventually contract a pair of vertices
    connected by an edge creating red self-loop.
 
    Notice that, for all $1 \le i < n$, contraction of rows $v_i$ and 
    $v_{i+1}$ results in a row containing red entry in all
    columns $v_j$ where $j < i$.
    Hence the red potential of $V(G_n)$ in $G_n$ is at least
    $\sum_{i = 1}^{n - 1}i-1 = \frac{(n - 1)(n - 2)}{2}$.
\qed\end{proof}

\section{Conclusions}\label{sec:conclu}

We have primarily studied bounded twin-width of certain graph classes which widely generalize proper interval graphs,
and have provided a straightforward procedures for constructing witnessing contraction sequences for them
assuming a suitable input representation of the graphs.
This study has been inspired by the fact that posets of bounded width are of bounded twin-width, and
that one can derive boundedness of twin-with of some simple generalizations of proper interval graphs (such as of
$k$-fold proper interval graphs \cite{DBLP:conf/focs/GajarskyHLOORS15}) from the former.
Our results in Section~\ref{sec:trans} can thus be seen as setting the limits of how far (in the class of graphs)
can posets of bounded width ``certify'' bounded twin-width.

Regarding the previous, we remark that it is considered very likely that the classes of graphs of bounded twin-width are not transductions of the classes of posets of bounded width
(although we are not aware of a published proof of this).
We think that the proper $k$-mixed-thin graph classes are, in the ``FO transduction hierarchy'', positioned strictly between the classes of posets of bounded width and the classes of bounded twin-width,
meaning that they are not transductions of posets of bounded width and they do not transduce all graphs of bounded twin-width.
We plan to further investigate this question.

Furthermore, Bonnet et al.~\cite{DBLP:journals/corr/abs-2102-06880} proved that the classes of structures of bounded twin-width are transduction-equivalent to the classes of permutations with a forbidden pattern.
It would be very nice to find an analogous asymptotic characterization with permutations replaced by the graphs of some natural graph property.
As a step forward, we would like to further generalize proper $k$-mixed-thin graphs while keeping the property of bounded twin-width.

\bibliography{k-mixed-thin}

\newpage

\end{document}